\documentclass[11pt]{amsart}

\usepackage{url}
\usepackage{latexsym}

\usepackage{amsmath,amssymb,amsthm,amsfonts,latexsym}

\usepackage{epsfig}
\usepackage{graphics}

\usepackage{hyperref}
\hypersetup{backref, colorlinks=true}

\usepackage{pstricks,pst-plot,pst-node,pst-text,pst-3d}

\let\=\partial

\newtheorem {pro}{Proposition}[section]
\newtheorem {thm}[pro]{Theorem}
\newtheorem {cor}[pro]{Corollary}
\newtheorem{lem}[pro]{Lemma}

\theoremstyle{definition}
 \newtheorem {rem}[pro]{Remark}
 
\newtheorem {dfn}[pro]{Definition}
\newtheorem {defs}[pro]{Definitions}
\newtheorem {exa}[pro]{Example}

\newcommand{\Q} {\mathbb{Q}}
\newcommand{\pa}{\partial}
\newcommand{\ws}{\overline{\Omega}}
\newcommand{\ko}{ k(0_+)}
\newcommand{\oi}{ \Omega_\infty ^j (X)}

\newcommand{\R} {\mathbb{R}}

\newcommand{\N} {\mathbb{N}}
\newcommand{\C} {\mathcal{C}}

\newcommand{\hn}{\mathcal{H}}

\newcommand{\db}{\overline{d}}
\newcommand{\St}{\mathcal{S}}
\newcommand{\xb}{X_{reg}}
\newcommand{\dbb}{\mathbb{D}}

\newcommand{\D}{\mathcal{D}}

\newcommand{\I}{\mathcal{I}}

\newcommand{\F}{\mathcal{F}}
\newcommand{\E}{\mathcal{E}}
\newcommand{\la}{\mathcal{L}}

\newcommand{\ep}{\varepsilon}
\newcommand{\K}{\mathcal{K}}

\title {{\bf   $L^1$ cohomology of bounded subanalytic manifolds}}



\author{Guillaume Valette}

\address
{Instytut Matematyczny PAN, ul. \'Sw. Tomasza 30, 31-027 Krak\'ow,
Poland} \email{gvalette@impan.pl}


\thanks{}


\begin{document}

\maketitle

\begin{abstract}
We prove some de Rham theorems on bounded subanalytic submanifolds
of $\R^n$ (not necessarily compact).  We show that  the $L^1$
cohomology of such a submanifold is isomorphic to its singular
homology.   In the case where the closure of the underlying manifold has only  isolated
singularities this implies that the $L^1$ cohomology is  Poincar\'e
dual to $L^\infty$ cohomology (in dimension $j <m-1$). In general,
  Poincar\'e
duality  is related to the so-called $L^1$ Stokes' Property. For oriented
manifolds, we show that the $L^1$ Stokes' property holds if and only
if integration realizes a nondegenerate pairing between $L^1$ and
$L^\infty$ forms. This is the counterpart of a theorem proved by
Cheeger on $L^2$ forms.
\end{abstract}

\section{introduction}
Given a Riemannian manifold $M$,  the $L^1$ forms are the
differential forms $\omega$ on $M$ satisfying
\begin{equation}\label{eq_l1_condition}\int_M |\omega|\, d vol_M
<\infty,\end{equation} where $|\omega|$ is the norm of the
differential form $\omega$ derived from the Riemannian metric of
$M$. The smooth $L^1$ forms having an $L^1$ exterior derivative constitute  a
cochain complex which gives rise to cohomology groups, {\bf the
$L^1$ cohomology groups of $M$}.

 In this paper we first prove a de Rham theorem for the $L^1$
cohomology:

\begin{thm}\label{thm_intro}Let $M$ be
a bounded subanalytic submanifold of $\R^n$. The $L^1$ cohomology of $M$ is
isomorphic to its singular cohomology.
\end{thm}

\medskip

Here, $M$  is  equipped  with the Riemannian metric inherited from the ambient
space. In particular, the $L^1$ cohomology groups are finitely generated and are topological invariants of $M$.

 %

Forms with integrability conditions have been the focus of interest of many authors. Let us mention, among many others,  \cite{ bgm,c1, c2, c3, cgm,d, weber,hp,s1,s2,y}. First, integration is necessary to construct a pairing, crucial to define a Poincar\'e duality morphism which we study below.  Secondly, integrability conditions are of foremost importance in geometric analysis and differential equations on manifolds. 

 The $L^1$ condition if of metric nature. The metric geometry of singularities is much more challenging than  the study of their topology.    For instance it is well known that
subanalytic sets may be triangulated and hence are locally homeomorphic to
cones.   This property is very important for it reduces the study of
the topology of the singularity to the study of the topology link. The story is more complicated is one is interested in the
description of the aspect of singularities from the metric point of
view. A triangulation may not be achieved without affecting drastically the metric structure the singular set. The proof of this theorem thus requires new techniques for we do not restrict ourselves  to metrically conical singularities. 

In \cite{vlt,vpams}, the author introduced and constructed some triangulations enclosing enough information to determine al the metric properties of the singularities. The idea was to control the way the metric  is affected by the triangulation. The proof of Theorem \ref{thm_intro} requires an acurate description of
the metric type of subanalytic singularities.
 Using the techniques developped in \cite{vlt} \cite{vpams} and \cite{vlinfty} we show that
the conical structure of subanalytic singularities is not only topological but Lipschitz in a
very explicit sense that we shall define in this paper. This is achieved in
section \ref{sect_lip} of this paper and it  is the keystone of the
proof of Theorem \ref{thm_intro}. This section is of its own
interest, offering a nice new description of the Lipschitz geometry of
subanalytic sets. We improve the results of  \cite{vlinfty} where it
was  shown that every subanalytic germ may be retracted in a
Lipschitz way (see also  \cite{sv}).

\medskip

The history of $L^p$ forms on singular varieties began when J.
Cheeger started constructing a Hodge theory for  singular compact
varieties. He first computed in \cite{c1,c2} the $L^2$ cohomology
groups for varieties with metrically  conical singularities. It
turned out to be related to  intersection cohomology making of it a
good candidate to get a generalized Hodge theory on singular
varieties \cite{c3,c4, c5, cgm}.

Since Cheeger's work on $L^2$ forms, many authors have investigated
$L^p$ forms on singular varieties \cite{bgm,d, weber,hp,s1,s2,y} (among many others).
Nevertheless all of them focus on particular classes of Riemmanian
manifolds, with strong restrictions on the metric near the
singularities, like in the case of the so-called $f$-horns or
metrically conical singularities. In the present paper we only
assume that the given set is subanalytic.

Recently, the author of the present paper computed the
$L^\infty$ cohomology groups for any subanalytic pseudomanifold.
Let us recall the de Rham theorem achieved in \cite{vlinfty}.

\begin{thm}\label{thm_intro_linfty}\cite{vlinfty} Let $X$ be a compact subanalytic pseudomanifold. Then, for any $j$:
$$H_\infty ^j(X_{reg}) \simeq I^{t}H^j (X).$$
Furthermore, the isomorphism is induced by the natural map
provided by integration on allowable simplices.
\end{thm}

 Here, $H_\infty^\bullet$ denotes the $L^\infty$ cohomology and $I^tH^j(X)$ the intersection cohomology of
 $X$ in the maximal perversity.
  The definitions of these cohomology theories are recalled in sections \ref{sect_ih} and \ref{sect_linfty} below.
  We write $X_{reg}$  for the nonsingular part of $X$,
  i. e. the set of points at which $X$ is a smooth manifold.

    Intersection homology was discovered by M. Goresky and R. MacPherson who
 computed these homology groups. What makes it  very attractive is that they showed in their fundamental
paper \cite{ih1} that it satisfies Poincar\'e duality for a quite
large class of sets (recalled in Theorem \ref{thm_poincare_ih}),
enclosing all the complex analytic sets (see also \cite{ih2}).

In view of the above paragraph, the two above de Rham theorems
raise the very natural question of  whether we can hope for
Poincar\'e duality between $L^1$ and $L^\infty$ cohomology.
Actually, the two above theorems, via Goresky and MacPherson's
generalized Poincar\'e duality, admit the following
corollary.

\begin{cor}\label{cor_poincare_duality_intro}
Let $X$ be an oriented subanalytic  pseudomanifold with isolated
singularities. Then, $L^1$ cohomology is Poincar\'e dual to $L^\infty$
cohomology in dimension $j<m-1$, i. e. for any  $j < m-1$:
$$H_{(1)} ^j(X_{reg}) \simeq  H_\infty^{m-j} (X_{reg}).$$
\end{cor}

More generally, if the singular locus is of dimension $k$ then the
$L^1$ cohomology is dual to the $L^\infty$ cohomology in dimension
$j <m-k-1$. This is due to the fact that in this case intersection
homology coincides with the usual homology of $X_{reg}$ (in
dimension $j$). Intersection homology turns out to be very useful to
assess the lack of duality between $L^1$ and $L^\infty$ cohomology.
We see that the obstruction for this duality to hold is of purely
topological nature. Although the $L^1$ and $L^\infty$ conditions are
closely related to the metric structure of the singularities, the
above theorems show that the knowledge of the topology of the
singularities is enough to enure Poincar\'e duality. It is worthy of
notice that the only data of the topology of $X_{reg}$ is not enough.

In his study of $L^2$ cohomology, Cheeger also pointed out a
problem that may arise on singular varieties, even with conical
singularities: the $L^2$ Stokes' property may fail. Roughly speaking, this property says that the exterior differential operator is
self-adjoint on $L^2$ forms (up to sign, considering $(m-j)$ -forms as the dual of $j$-forms, see (\ref{eq_Stokes'_intro})).
This property is crucial
in Hodge theory, which yields Poincar\'e duality as a byproduct.
Cheeger investigated the case of conical singularities in
\cite{c2} and completely clarified the situation. He showed that
the $L^2$ Stokes' property holds on conical singularities if and only if Poincar\'e
duality holds. Thus, in this case, a nice Hodge theory may be
performed and Cheeger was able to prove that every cohomology
class has a unique harmonic representative.
Cheeger's  $L^2$ Stokes' property  is also crucial because it
allows to define a pairing on the $L^2$ cohomology groups by
integrating wedge products of forms. The Poincar\'e duality isomorphism on  $L^2$ cohomology then
results from this pairing which provides a very natural isomorphism.

The $L^p$ Stokes'
property has been then studied by Y. Youssin  on $f$-horns in
\cite{y}, who obtained an analogous result.

Therefore, in our framework,  the latter duality for $L^1$ cohomology very naturally
raises the question on whether the $L^1$ Stokes property holds and
whether integration provides an isomorphism between $L^1$ and
$L^\infty$ cohomology. In order to be more specific, let us
explicitly define the {\bf $L^1$ Stokes' property} by saying that
it holds (in dimension $j$) on a $C^\infty$ manifold $M$  of
dimension $m$ whenever for any $C^\infty$ $L^1$ $j$-form $\alpha$,
with $d\alpha$ $L^1$ we have:
\begin{equation}\label{eq_Stokes'_intro}\int_M \alpha \wedge d\beta
=(-1)^{j+1} \int_M d \alpha \wedge \beta,\end{equation} for any
$L^\infty$ $(m-j)$-form $\beta$ with $d\beta$ $L^\infty$.

For smooth forms on compact manifolds without boundary this is
always true by Stokes' formula. Somehow, the question is whether
the singularities behave like a boundary or if the closure of $M$
may behave like a manifold. This question occurs especially in the
case where the singular locus of the closure of $M$ is of low
dimension.

We shall answer this question in a very precise way, giving a
$L^1$ counterpart of Cheeger's theorem on the $L^2$ Stokes'
property.   Again,
our theorems on $L^1$ cohomology  hold for any
subanalytic bounded manifold (metrically conical or not).

Given a submanifold $M\subset \R^n$ we shall write $\delta
M$ for the set $cl(M)\setminus M$, where $cl(M)$ stands for the
topological closure of $M$. We shall prove:

\begin{thm}\label{thm_intro_l1_stokes_property} Let $j<m$ and let $M$ be a bounded subanalytic oriented
manifold.   The $L^1$ Stokes' property holds for $j$-forms iff
$\dim \delta M <m-j-1$.
\end{thm}

 In particular, if $ cl(M)$ has only isolated singularities then the
$L^1$ Stokes' property holds in any dimension $j<m-1$. In this case,
integration of forms induces the Poincar\'e duality isomorphism of
Corollary \ref{cor_poincare_duality_intro}.

Noteworthy, the obstruction for the $L^1$ Stokes property  to
hold is also purely topological. The only knowledge of the dimension of
the singular locus is enough to ensure that this property holds,
no matter how fast the volume is collapsing near the
singularities.

\subsection*{Dirichlet $L^1$ cohomology.} Let $M\subset \R^n$ be a subanalytic bounded submanifold. We just explained that in
the case of non closed oriented manifolds, the $ L^1$-Stokes' property may fail.   This  "boundary phenomenon"
may appear near the singularities  preventing the $L^1$ classes from being
Poincar\'e dual to the $L^\infty$ classes.

On compact  manifolds with boundary, "ideal boundary conditions"
are usually put in order overcome this kind of problems. They
give rise so-called {\bf Dirichlet cohomology}. The Dirichlet
forms are those whose restriction to the boundary is identically
zero. These are also the forms satisfying the $L^2$ Stokes'
property.

In our setting, if $\omega$ is a form defined on  $M$, it does not
make sense to require that it vanishes on $\delta{M}$. {\bf Dirichlet
$L^1$ cohomology} is thus usually defined (see for instance
\cite{iw}) as the cohomology of the  $L^1$ forms $\alpha$ (with
$d\alpha$ $L^1$) satisfying (\ref{eq_Stokes'_intro}). This is the
biggest space of $L^1$ forms on which $d$ is self-adjoint (up to
sign, identifying $L^\infty$ with the dual of  $L^1$). This cohomology theory is
discussed in section \ref{sect_l1sp}.

We will denote the Dirichlet $L^1$  cohomology of a submanifold
$M\subset \R^n$ by $H_{(1)}^{m-j}(M;\delta M)$. Now, as in the case of
manifolds with boundary, Lefschetz-Poincar\'e duality holds in
general:

\begin{thm}\label{thm_intro_poincare_dirichlet}
For any bounded  subanalytic orientable submanifold
$M\subset \R^n$:$$H_{(1)}^j(M;\delta M) \simeq H^{m-j}_{\infty}(M).$$
\end{thm}

It is worthy of notice that this duality is a general fact on
bounded subanalytic manifolds: we do not assume that the closure
of $M$ is a pseudomanifold. The version stated in Theorem
\ref{thm_Poincare duality_dirichlet} is actually even stronger.

In particular, by Goresky and MacPherson's generalized Poincar\'e
duality,  the Dirichlet $L^1$ cohomology is isomorphic to
intersection homology in the zero perversity and Theorem
\ref{thm_intro_linfty} and Theorem
\ref{thm_intro_poincare_dirichlet} admit the following immediate
interesting corollary.

\begin{cor}\label{cor_dirichlet_de_rahm_intro}(De Rham theorem for Dirichlet $L^1$ cohomology)
Let $X$ be a subanalytic bounded orientable pseudomanifold.  We have:
$$ H_{(1)}^{j}(X_{reg};X_{sing})\simeq I^0 H^j(X_{reg}).$$
\end{cor}

Here $X_{sing}$ stands for the singular locus and $X_{reg}$
denotes its complement in $X$.


\subsection*{Content of the paper.}
   Section \ref{sect_lip} introduces and yields  the "Lipschitz conic structure of subanalytic
   sets" (definition \ref{dfn conical}). We prove in section \ref{sect_weakly_l1} some basic
   results on $L^1$ cohomology and  establish Theorem
   \ref{thm_intro} in section \ref{sect_de_rham_l1}. Poincar\'e duality for $L^1$ cohomology is
   then discussed in section \ref{sect_cor_poinc}.
   In section \ref{sect_dirichlet}  we introduce the $L^1$ Dirichlet cohomology groups and
   establish the relative form of Lefschetz-Poincar\'e duality claimed in
   Theorem \ref{thm_intro_poincare_dirichlet}. We then study  the $L^1$ Stokes'
   property, proving Theorem
   \ref{thm_intro_l1_stokes_property}
   in  section \ref{sect_l1sp}. We end this paper with a concrete
   example, the suspension of the torus, on which we discuss all
   the results of this paper.

\subsection*{Acknowledgement}This paper was started while the author was a participant of the semester on O-minimal structures and real algebraic geometry
at the  Fields institute of Toronto and carried out in Cracow
while the author was a researcher for the Polish Academy of
Science. The author wishes to  thank these two institutions for
their hospitality. It is also his pleasure to thank Andrzej Weber,
Pierre Milman, Wies\l aw Paw\l ucki  and Jean-Paul Brasselet  for
valuable discussions.

\subsection*{Notations and conventions.} In the sequel, all the considered sets and maps will be
subanalytic (if not otherwise specified) except the differential forms. 

 By
"{\bf subanalytic}" we mean "globally subanalytic", i. e. which remains
subanalytic after  compactifying  $\R^n$ (by $
\mathbb{P}^n$).

 Given a set
$X\subset \R^n$, we denote by $C^j(X)$ the singular cohomology
cochain complex and by $H^j(X)$ the cohomology groups. Simplices are defined as
continuous (subanalytic) maps $\sigma :\Delta _j \to X$, where
$\Delta_j$ is the standard simplex.

Given two nonnegative functions $\xi:X\to \R$ and $\eta:X\to  \R$ we
will write $\xi \sim \eta$ if there is a positive constant $C$
such that $\xi \leq C\eta$ and $\eta \leq C\xi$. We write
$[\xi;\eta]$ for the set $\{x \in X\times \R: \xi(x)\leq y\leq \eta(x)\}$
and define similarly the open interval $(\xi;\eta)$.

Given a (subanalytic) set $X$, we denote by $X_{reg}$ the set of
point near which $X$ is a $C^\infty$ manifold  and by $X_{sing}$,
its complement in $X$. The subsets $\delta M$ and $cl(M)$ are also
as explained above. By manifold we will mean $C^\infty$ manifold.

We shall say that a function $\xi:X\to \R$ is {\bf Lipschitz} if there
is a constant $C$ such that for any $x$ and $x'$ in $X$:
$$|\xi(x)-\xi(x')|\leq C |x-x'|.$$
A map $f:X\to \R^k$ is Lipschitz if all its components are
Lipschitz and a homeomorphism $h$ is {\bf bi-Lipschitz} if both $h$ and
$h^{-1}$ are Lipschitz.

We shall write $S^{n-1}(x_0;\ep)$ for the sphere of radius $\ep>0$
centered at $x \in \R^{n}$  and $B^n(x_0;\ep)$    for the
corresponding ball.  We will write $L(x_0;X)$ for the  {\bf link} of $X$
at $x_0$. It is the subset $S^{n-1}(x_0;\ep)\cap X$, where $\ep>0$
is small enough. By \cite{vpams} this subset is, up to a subanalytic bi-Lipschitz map, independent of $\ep>0$.

\section{On the Lipschitz geometry of subanalytic sets.}\label{sect_lip}
 The results of this section  will be very
important to compute the $L^1$  cohomology groups
later on.

It is well known that subanalytic sets are locally homeomorphic to
cones. It is not true that subanalytic germs of singularities are
bi-Lipschitz homeomorphic to cones.
 We describe the metric types of subanalytic
  germs in a very precise way. This is very important  since the $L^1$
  condition heavily relies on the metric. Roughly speaking, we  show that, given a subanalytic germ $X$,
   we can find a subanalytic  homeomorphism from a cone  (over the link) such that the eigenvalues of the
   pullback of the metric induced by $\R^n$ on $X$ by
   this homeomorphism are increasing as we are wandering away  from the origin.
   This improves significantly the results of \cite{vlt} \cite{vlinfty} where
   a Lipschitz strong deformation retraction onto the origin was
    constructed.

Given  $n>1$ and a positive constant $R$ we set:
$$\C_n(R):=\{(x_1;x') \in \R\times \R^{n-1}:  0 \leq |x'| \leq R x_1\,\}.$$
For $n=1$, we just define $\C_1$ as the positive $x_1$-axis.

\subsection{Regular lines.} We start by recalling a result of
\cite{vlinfty}.

\begin{dfn}\label{boule  reguliere}
Let $X$ be a   subset of $\R^{n}$. An element $\lambda$ of
$S^{n-1} $ is said to be {\bf regular  for $X$} if there is a
positive number $\alpha$ such that:
$$dist(\lambda;T_x X_{reg}) \geq \alpha,$$
for any $x$ in $X_{reg}$.
\end{dfn}

Regular lines do not always exist, as it  is shown by the simple
example of a  circle. Nevertheless, given a subanalytic set of
empty interior, up to a bi-Lipschitz homeomorphism,  we can  get a
line which is regular.
 This is
what is established by  theorem $3.13$  of \cite{vlt}. This
theorem has then been improved in \cite{vlinfty} into a statement
that we shall need in its full generality. It is recalled in Lemma
\ref{prop proj reg}. To state this lemma,  we need  the
following  definition.

\begin{dfn}
Let $A, B \subset \R^n$. A  map $h:A \to B$  is {\bf $x_1$-preserving} if it
preserves the first coordinate in the canonical basis of $\R^n$.
\end{dfn}


We denote by $\pi_n:\R^{n}\to \R^{n-1}$ the canonical projection. In the Lemma below all the considered germs are germs at the origin.

\medskip

\begin{lem}\label{prop proj reg}\cite{vlinfty}
Given germs $X_1,\dots,X_s \subset \C_{n}(R)$, there exist a germ of
$x_1$-preserving bi-Lipschitz
  homeomorphism $h:\C_n(R) \to \C_{n}(R)$, with $R>0$, and a cell decomposition $\mathcal{D}$ of $\R^n$ such
  that:
  \begin{enumerate}
  \item  $\D$ is compatible with $h(X_1),\dots, h(X_s)$ \item $e_n$ is regular for any cell of $\D$ in $\C_{n}(R)$
   which is a graph over a cell of $\C_{n-1}(R)$ of $\D$
\item Given finitely many   germs of nonegative functions
$\xi_1,\dots,\xi_l$ on $\C_{n}(R)$, we may assume that on each
cell $D$ of $\mathcal{D}$, every germ  $\xi_i\circ h$ is $\sim$
to a function of the form:\begin{equation}\label{eq
prep}|y-\theta(x)|^r a(x)\end{equation} (for $(x;y) \in \R^{n-1}
\times \R$) where
$a,\theta:\pi_n(D) \to \R$ are   functions with $\theta$ Lipschitz and $r \in \Q$. 
\end{enumerate}
\end{lem}

\begin{rem}\label{rmk graphes en plus}
 Given a family of    Lipschitz functions $f_1, \dots, f_k$
defined over $\R^n$ we can find some Lipschitz functions
$\xi_1\leq \dots\leq   \xi_l$ and a cell decomposition $\D$ of
$\R^{n-1}$ such that over each cell $D \in \R^{n-1}$ delimited by
the graphs of two consecutive functions $[\xi_{i|D};\xi_{i+1|D}]$,
with $D\in \D$, the functions $|q_{n+1}-f_i(x)|$ (where
$q=(x;q_{n+1})$) are comparable to  each other (for relation
$\leq$)
 and comparable to the functions $f_i \circ \pi_n$. Indeed, it suffices to choose a cell decomposition $\D$ compatible with the sets $f_i=f_j$ and to  add the graphs of the
 functions   $f_i$, $f_i+f_j$  and  $\frac{f_i+f_j}{2}$. We may
 then use $\min$ and $\max$ to transform this family into an
 ordered
 family (for $\leq$).
 \end{rem}

\subsection{Lipschitz conic structure  of subanalytic sets.} This section is crucial in the proof of our de Rham theorems. We introduce and establish  what we call "the Lipschitz conic structure" of subanalytic sets.

  Let  $X\subset \R^n$ of dimension $m$  and let $x_0\in cl(X) $.

\begin{dfn}
A {\bf tame basis} on a manifold $M$ is a basis $\lambda
_1,\dots,\lambda _m$ ($m=\dim M$) of bounded subanalytic
$1$-forms on $ M$  such that:
\begin{equation}\label{eq_tame_basis}|\wedge_{i=1} ^m \lambda_i|\geq \ep >0,
\end{equation} on $M$.
%
\end{dfn}

Let us make a point that we do not assume the tame bases to be
continuous, but, as they are assumed to be subanalytic, they are
indeed implicitly required to be smooth almost everywhere. This
will be enough for us, since, for integrability conditions, only the behavior almost everywhere is relevant. Alike, in the
definition below, the $\varphi_i$'s do  not need to be continuous, but
indeed only the generic values of these functions really matter
since (\ref{item_dfn_metric_conical}) of the definition below is required almost everywhere. We shall also pull-back the forms via subanalytic maps. The pullback will be well defined almost everywhere since, once again, subanalytic mappings are smooth generically.

\begin{dfn}\label{dfn conical}
We say that  $X$ is {\bf Lipschitz conical} at
$x_0 \in X$ if there exist a positive real number $\ep$ and a
 Lipschitz homeomorphism
$$h: (0;\ep) \times L(x_0;X)\to X \cap B^n(x_0;\ep) \setminus \{x_0\},$$ with
$d(x_0;h(t;x))=t$, such that we can find some positive functions $
\varphi_1,\dots,\varphi_{m-1}:(0;\ep) \times  L(x_0;X) \to \R$, for
which we have:
\begin{enumerate} \item\label{item_dfn_phi_decreasing} The $\varphi_i(t;x)$'s are decreasing to
zero as $t$ is going to zero for any $x$, \item The $\varphi_i(t;x)$'s
are bounded below  on any closed set disjoint from $
\{t=0\}$.\item\label{item_dfn_metric_conical} There is
 a tame basis
$ \lambda_1,\dots,\lambda_{m-1}$ of $L(x_0;X_{reg})$ such that if $\theta_i:=h^{-1*}( \varphi_{i}\cdot \lambda _i )$ then $(h^{-1*}dt; \theta_1;\dots;\theta_{m-1})$ is a tame basis on a dense subset of $X_{reg}$. \end{enumerate}
\end{dfn}

\medskip

\begin{thm}\label{thm Lipschitz conic structure}
 Every  (subanalytic) set  is Lipschitz
 conical at any point.
\end{thm}

\begin{proof}
 We shall consider  sets $A\subset
\R^n$ as families parameterized by $x_1$ and write $A^\ep$ for the "fiber" at $\ep \in \R$, that is to say:
$$A^\ep:=\{x \in \R^{n-1}: (\ep;x) \in A\}.$$


 We will
actually prove by induction on $n$ the following statements.

\bigskip

{\bf$(\textrm{A}_n)$} Let $X_1,\dots,X_s$ be finitely many subsets
of $\C_n(R)$ and let $\xi_1,\dots,\xi_l$ be some bounded functions.

There exist  positive real numbers $R$ and  $\ep$, together with a
Lipschitz $x_1$-preserving homeomorphism
$$h:(0;\ep) \times B^{n-1}(0;R) \to \C_n(R)  \setminus \{0\},$$   such that
 for
every  $j\in \{1,\dots,s\}$, we can find some positive functions
$\varphi_{1,j}, \dots,\varphi_{\mu_j-1,j}$ ($\mu_j:=\dim X_j$) on
$(0;\ep) \times X_{j,reg}^\ep $ with:
\begin{enumerate}\item
$h((0;\ep) \times X_{j}^\ep)=X_{j} \cap \{0< x_1<\ep\} $
\item\label{item_phi_decreasing} The $\varphi_{i,j}(t;x)$'s are
decreasing to zero as $t$ goes to zero, for any $x \in X_{j,reg}^\ep$
  \item\label{item_bounded_below} The $\varphi_i(t;x)$'s
are bounded below  on any closed set disjoint from $\{t= 0\}$
\item\label{item_metric_conical}
 There is
 a tame basis
$ \lambda_{1,j},\dots,\lambda_{m-1,j}$ of $X^\ep _{j,reg}$ such that if $\theta_{i,j}:= h^{-1*}( \varphi_{i,j}\cdot \lambda _{i,j}) $ then $(h^{-1*}dt; \theta_{1,j};\dots;\theta_{m-1,j})$ is a tame basis of a dense subset of $X_{j,reg}$.
\item \label{item_tilda_decreasing} There is a constant $C$ such
that for any $i \leq l$ and any $ 0< \tau  \leq u \leq t$ we
have:\begin{equation}\label{eq decroissance fn up to
contant}C_\tau \xi_i(h(\tau;x)) \leq \xi_i(h(u;x)) \leq C \xi_i(h(t;x)).\end{equation} for some positive constant $C_\tau$.
 \end{enumerate}

\bigskip

Before proving these statements, let us make it clear that this
implies the desired result. Let $X\subset \R^n$. We can assume that $0\in X$ and  work nearby the origin.  The set $$\hat{X}:=\{(x_1;x) \in \R \times X: |x|=x_1\} $$ is a subset of $\C_{n+1}(R)$ (for $R>1$) to
which we can apply {\bf$(\textrm{A}_{n+1})$}. Observe that
$\hat{X}$ is bi-Lipschitz equivalent to $X$. This
means that it is enough to check the properties $(1-3)$ of
definition \ref{dfn conical} for $\hat{X}$. But they are implied
by $(\ref{item_phi_decreasing})$, $(\ref{item_bounded_below})$ and
$(\ref{item_metric_conical})$ of {\bf$(\textrm{A}_{n+1})$}.

The assertion (\ref{item_tilda_decreasing}) is not necessary to
prove that $X$ is Lipschitz conical. It is assumed so as to
perform the proof of (\ref{item_phi_decreasing}) during the
induction step.
\bigskip

As {\bf$(\textrm{A}_n)$} obviously holds in the case where $n=1$
($h$ being the identity map), we fix some $n>1$. We fix some subsets $X_1,\dots,X_s$ of $\C_n(R)$, for $R>0$,  and
some subanalytic bounded functions $\xi_1,
\dots,\xi_l:\C_n(R)\to\R$.

 \medskip



\medskip

Apply Lemma \ref{prop proj reg} to the family constituted by the
$X_i$'s and the union of the zero loci of the $\xi_i$'s.  We get a
$x_1$-preserving bi-Lipschitz map  $h:\C_n(R)\to \C_n(R)$ and a
cell decomposition $\D$ such that $(1)$, $(2)$, and $(3)$ of the
latter lemma hold. As we may work up to a $x_1$-preserving
 bi-Lipschitz map we will identify $h$ with the identity map. Hence, thanks to $(3)$ of the latter Lemma, we
may assume that the functions $\xi_i $'s
 are $\sim$ to a function like in (\ref{eq prep}).


Let $\Theta$  be a  cell of $\D$ in $\C_n (R)$ which is the graph of a
 function $\eta:\Theta' \to \R$, with $\Theta'\in \D$.  By $(2)$ of Lemma
\ref{prop proj reg}, $\eta$ is then necessarily a  Lipschitz
function. Consequently, it may be extended to a
Lipschitz function on the whole $\C_{n-1}(R)$ whose graph still lies in $\C_n (R)$. Repeating this for all the
cells of $\D$ which are graphs over a cell of $\D$ in $\R^{n-1}$
we get a family of functions $\eta_1,\dots,\eta_v$. Using  the
operators $\min$ and $\max$ we may transform this family in an
ordered one (for $\leq$), so that, keeping the same notations for
the new family, we will assume that it satisfies $\eta_1\leq \dots
\leq \eta_v$.


Fix an integer $1 \leq j < v $ and a connected component $B$ of
$(\eta_j;\eta_{j+1})$. Let $\Theta$ be a cell of $\D$ and set for simplicity $D:=\Theta \cap  B$. 

Up to
constants, the functions $\xi_k $'s
 are like in (\ref{eq prep}) on $D$,  i. e. there exist $(n-1)$-variable  functions on $D$, say $\theta_k$ and $a_k$,
 $k=1,\dots,m$ with $\theta_k$ Lipschitz such that: $$\xi_k(x;y) \sim (y-\theta_k(x))^{\alpha_k}
a_k(x),$$ for $(x;y)\in D \subset \R^{n-1}\times \R$.

We shall apply the induction hypothesis to all the $a_k$'s
(obtained for all such sets $D$). Unfortunately this is not enough
if one wants to get  that  the $\xi_k$'s satisfy $(5)$, due to the
term $(y-\theta_k)$ in the decomposition of the $\xi_k$'s just
above. Therefore,  before applying the induction hypothesis, we
need to complete the family to which we will apply $(5)$ of the
induction hypothesis by some extra bounded $(n-1)$ variable
functions  that we are going to introduce.

As the zero loci of the $\xi_k$'s are included in the graphs of the $\eta_i$'s, we
have on  $D$ for every $k$, either
  $\theta_k \leq \eta_j$ or $\theta_k \geq \eta_{j+1}$. We will assume for simplicity that $\theta_k \leq \eta_j$.

This means that for $(x;y)\in D \subset \R^{n-1}\times \R$:
\begin{equation}\label{eq min}\xi_k(x;y) \sim \min ((y-\eta_j(x)) ^{\alpha_k} a_k(x)
;(\eta_j-\theta_k(x)) ^{\alpha_k}  a_k(x)),\end{equation}  if
$\alpha_k$ is negative and
\begin{equation}\label{eq max}\xi_k(x;y) \sim \max ((y-\eta_j(x)) ^{\alpha_k} a_k(x)
;(\eta_j-\theta_k(x)) ^{\alpha_k}  a_k(x)),\end{equation} in the
case where $\alpha_k$ is nonnegative.

 First, consider the following functions:
\begin{equation}\label{eqdefkappak}\kappa_k(x):=(\theta_{k}(x)-\eta_j(x))^{\alpha_k} a_k(x), \qquad k=1,\dots ,l.\end{equation}
For every $k$,  the function $\kappa_k$ is  bounded for it is
equivalent to the function $\xi_k(x;\eta_{j}(x))$ which is bounded since $\xi_k$ is. 
Complete the family $\kappa$ by adding the functions $(\eta_{j+1}-\eta_j)$ as well as the functions  $\min(a_k;1)$. The union of all these families (the just obtained family $\kappa$ depends on $D$), obtained for every  such set $D$ (intersection of a
connected component of  $(\eta_j;\eta_{j+1})$, for some $j$, with some cell $D$ of $\D$) provides us a
finite collection of  functions $\sigma_1,\dots,\sigma_p$.

We now turn to the construction of the desired homeomorphism.
The cell decomposition $\D$ induces a cell decomposition of $\R^{n-1}$. Refine it into a cell decomposition $\E$ compatible with the  zero loci of the functions $(\eta_j-\eta_{j+1})$. Apply induction hypothesis to the family constituted by the
cells of $\E$ which lie in $\C_{n-1}(R)$. This
provides a  homeomorphism   $$h:(0;\ep) \times B^{n-2}(0;R) \to
\C_{n-1}(R) .$$

We first are going to lift $h$ to a homeomorphism
$\tilde{h}: (0;\ep) \times B^{n-1}(0;R) \to \C_{n}(R) $.

 Thanks to
the induction hypothesis, we may assume that the functions
$\sigma_1,\dots,\sigma_p$ 
 satisfy (\ref{eq decroissance fn up to
contant}).

We lift $h$ as follows. For simplicity we define $\eta_j'$ as the
restriction of $\eta_j$ to $\C_n(R)\cap \{x_1=\ep\}$. On
$(\eta_j;\eta_{j+1})$ we set  
$$\nu(q):= \frac{y-\eta_{j}(x)}{\eta_{j+1}(x)-\eta_{j}(x)},$$
where $q=(x;y) \in \R^{n-1}\times \R$. Then, for  $(t;q) \in  (0;\ep) \times (\eta_j';\eta_{j+1}')$
$$\widetilde{h}(t;q):=(h(t;x);\nu(q)(\eta_{j+1}(h(t;x))-\eta_j(h(t;x))) +\eta_j(h(t;x))).$$

In virtue of  the induction hypothesis,  the inequality (\ref{eq
decroissance fn up to contant}) is fulfilled by the functions
$(\eta_{i+1}-\eta_i)$.
 Therefore, as $h$ is Lipschitz, we see that $\widetilde{h}$ is Lipschitz as
 well. As $(1)$ holds by construction for every cell, it holds for
 all the $X_j$'s.

\bigskip

We now turn to define the functions $\varphi_{i,j}$. Actually, as
all the $X_i$'s are unions of cells, it is enough to carry out the
proof on every cell $E\in \E$, i. e. to define some functions
$\varphi_{1,E},\dots,\varphi_{\mu-1,E}$ (where $\mu=dim \, E$),
decreasing to $0$ with respect to $t$, and a tame basis
$\lambda_{1,E},\dots,\lambda_{\mu-1,E}$ such that the family $(\widetilde{h}^{-1*}dt; \theta_{1,E};\dots;\theta_{\mu-1,E})$,
where $\theta_{i,E}:=\widetilde{h}^{-1*}(\varphi_{i,E}\cdot\lambda_{i,E})$, is a tame basis of $E$.

Indeed, the desired functions $\varphi_{i,j}$ can then be defined as the
functions induced by all the functions $\varphi_{i,E}$ (defined on
$\widetilde{h}^{-1}(E)$), for all the cells $E$ of dimension
$\mu_j$ included in $X_j$ (as pointed out before definition
\ref{dfn conical} only the generic values of $\varphi_{i,j}$
actually matter).

Fix a cell $E\subset \C_n(R)$, set $E':=\pi(E)$ and $\mu':=dim \,E'$,
where $\pi:\C_n(R)\to \C_{n-1}(R)$ is the obvious orthogonal
projection. Let now $\varphi_{1,E'},\dots,\varphi_{\mu'-1,E'}$ be
the functions given by the induction hypothesis. We distinguish two cases.

\underline{{\it First case}}: $\mu'=\mu-1$ (where
$\mu=dim \,E$).  Let us set:
$$\varphi_{i,E}(t;x):=\varphi_{i,E'}(t;\pi(x)).$$
 As $\mu'=\mu-1$, the cell  $E$ is included in
$[\eta_{j|E'};\eta_{j+1|E'}]$, for some $j\leq \lambda$, and we
also set:
\begin{equation}\label{eq_def_varphi_i}
 \varphi_{\mu-1,E}(t;x):=\frac{\eta_{j+1}(h(t;x))-\eta_j(h(t;x))}{\eta_{j+1}(h(\ep;x))-\eta_j(h(\ep;x))},
\end{equation}
 Let us show that
these functions satisfy  $(\ref{item_phi_decreasing})$ and $(\ref{item_bounded_below})$.

 Recall that we applied $(\ref{item_tilda_decreasing})$ of the induction hypothesis to the function $(\eta_{j+1}-\eta_j)(x)$.
If a function $\xi$ satisfies
$(\ref{eq decroissance fn up to contant})$ then $$\xi(h(s;x)) \sim
\inf_{s \leq t< \ep} \xi(h(t;x)),$$ and consequently $\xi\circ h$ is
$\sim$ to an increasing function. 

 Therefore,
 changing $\varphi_{i,E}$ for an
equivalent function if necessary, we may assume that it is
increasing with respect to $t$. As the graphs of the $\eta_i$'s are
included in $\C_n(R)$, the $\eta_i$'s must vanish at the origin.
Consequently $\varphi_{\mu-1,E}$ tends to zero, as $t$ goes to
zero for any $x \in E$, which yields
$(\ref{item_phi_decreasing})$.

   As $(\eta_{j+1}-\eta_j)$ satisfy (\ref{eq decroissance fn up to contant}), the $\varphi_i$'s are bounded away from zero on $ (\tau;\ep)\times E^\ep$ for every $0< \tau <\ep $, showing (\ref{item_bounded_below}).

\bigskip

We now are going to define our tame basis of $1$-forms $\lambda_{i,E}$ in order to prove (\ref{item_metric_conical}).

Denote by $\pi_E:E \to E'$ the restriction of the orthogonal
projection.  Let us now set on $(0;\ep)  \times
E^\ep$
\begin{equation}\label{eq_lambda_iE}\lambda_{i,E}:=\pi_{|E}^*
\lambda_{i,E'}.\end{equation}




 Then set for $x\in E'$
and $a \in [0;1]$:
$$\eta_{j,a}(x)=(\eta_{j+1}(x)-\eta_j(x))a +\eta_j(x).$$

Denote by $E_a$ the graph of $\eta_{j,a}$.

Put now  $$\lambda_{\mu-1,E}(q)(u)=0,$$ if $u$ is tangent to
$(E_{\nu(q)})^\ep$, and   finally set
$$\lambda_{\mu-1,E}(q)(e_n)=1.$$

 As the
 $\eta_{i,a}$ are Lipschitz with a Lipschitz constant bounded with respect to $a$, the angle between $e_n$ and the tangent to the
  graph of $\eta_{i,\nu(x)}$ is bounded below away from zero,  and therefore the
  norm of $\lambda_{\mu-1,E}$ is bounded.    The Lipschitz character of the $\eta_{j,a}$'s also implies that the family $\lambda_{1,E},\dots,\lambda_{\mu-1,E}$ is a tame basis of $E^\ep$.

   By definition of $\widetilde{h}$ and $\varphi_{\mu-1,E}$ we have $d_{(t;x)} h(e_n)\sim \varphi_{\mu-1,E}(t;x)$ so that:
    $$|\widetilde{h}^{-1*} \lambda_{i,E}|\sim \frac{1}{\varphi_{\mu-1,E}}.$$

 The forms  $\theta_{i,E}:=\widetilde{h}^{-1*}(\varphi_{i,E}\cdot \lambda_{i,E})$ are thus bounded. For the same reasons as for the $\lambda_i$'s, the family
$(\widetilde{h}^{-1*}dt; \theta_{1,E};\dots;\theta_{\mu-1,E})$
 is a tame basis of $E$.

%
%
%

\medskip

\underline{\textit{Second case}}: $\mu=\mu'$.  In this case we only have to define $(\mu'-1)$ functions and $(\mu'-1)$ $1$-forms.
This may be done like in the first case (like in (\ref{eq_def_varphi_i}) and (\ref{eq_lambda_iE})). This is indeed much easier to check that  $(2)$ $ (3)$ and $(4)$ hold,
 since,
as $\pi_E$ is bi-Lipschitz, the required properties which are true
downstairs for $h$ thanks to the induction hypothesis obviously
continue to hold upstairs for $\tilde{h}$. This completes the proof of $(2)$ $ (3)$ and $(4)$.



 Finally, we
have to check that the $\xi_k$'s fulfill (\ref{eq decroissance fn
up to contant}) for $\widetilde{h}$. As the $\xi_k$'s are  bounded
this is enough to check it for the functions $\min(\xi_k;1)$. We
check it on a given cell $E\in \E$. Fix an integer $1 \leq k \leq
l$. By the induction hypothesis we know that the $\kappa_i$'s (see
(\ref{eqdefkappak})) satisfy (\ref{eq decroissance fn up to
contant}). Remark that the function $\nu(\widetilde{h}(t;q))$ is
constant with respect to $t$.

   Observe that by (\ref{eq min}) and (\ref{eq max})
  it is enough to show that the functions $\min((y-\eta_j(x))^{\alpha_k} a_k(x);1)$ and the functions $\min(|\theta_k-\eta_j|(x) ^{\alpha_k} a_k(x);1)$ satisfy (\ref{eq decroissance fn
up to contant}).
 As for the latter functions this follows from the
induction hypothesis and choice of the $\kappa_i$'s, we only need
to focus on the former ones.

For simplicity we set $$F(x;y):=(y-\eta_j(x))^{\alpha_k} a_k(x),$$
and $$G(x):=(\eta_{j+1}-\eta_j)(x)^{\alpha_k} \cdot a_k(x).$$

We have to  show the desired inequality for $\min(F;1)$. We have:
\begin{equation}\label{eq F G}F(q)=\nu(q)^{\alpha_k} \cdot G(x),\end{equation}
where again $q=(x;y)$.

As $\nu(\widetilde{h}(t;q))$ is
constant with respect to $t$, this
implies that:
\begin{equation}\label{eq2 F G}F(\widetilde{h}(t;q))=\nu(q)^{\alpha_k} \cdot G(h(t;x))\end{equation}

We assume first that $\alpha_k$ is negative. Thanks to the
induction hypothesis we know that for $0< \tau \leq u \leq t$:  $$C_\tau  \min (G(h(\tau;x));1)\leq \min(G(h(u;x));1)\leq C \min
(G(h(t;x));1),$$
for some positive constants $C_\tau,C$. 

But this implies (multiplying by $\nu^{\alpha_k}$ and applying
(\ref{eq F G}) and (\ref{eq2 F G})) that
$$C_\tau \min(F(\widetilde{h}(\tau;q));\nu^{\alpha_k}(q);1)\ \leq  \min(F(\widetilde{h}(u;q));\nu^{\alpha_k}(q);1)\leq C \min (F(\widetilde{h}(t;q));\nu^{\alpha_k}(q);1). $$
But, as $\alpha_k$ is negative,
$\min(F;\nu^{\alpha_k};1)=\min(F;1)$ and  we are done.

We now assume that $\alpha_k$ is nonnegative. This implies that $F $ is a bounded function (by (\ref{eq max})). Moreover, by (\ref{eq
max}), it is enough to show the desired inequality for $F $, 
and thanks to (\ref{eq F G}), it actually suffices to show it for $G$.
As $G$ is one of the $\kappa_i$'s,  the result follows from the induction hypothesis. This yields
(\ref{eq decroissance fn up to contant}) for $\widetilde{h}$,
establishing (\ref{item_tilda_decreasing}).
%
%
\end{proof}

\begin{rem}\label{rem_conical}
\label{rmk_conic_structure} \begin{enumerate} \item  As in
\cite{vlt}, the $\varphi_i$'s could be expressed as quotients of
sums of products of powers of the monomial $t$ and distances to
some subsets of the link.\item Observe that in the proof of the above
the induction hypothesis is stronger than the theorem since we
have proved the Lipschitz conic structure of finitely many sets
simultaneously and that the homeomorphism is defined on the
ambient space as well.
\item \label{rem_conical_item3}
   Denote by $\rho_X$
    the Riemannian
   norm induced by the ambient space on $ X_{reg}$. Condition $(3)$ of definition \ref{dfn conical}  clearly implies the following:
\begin{equation}\label{eq_dfn_conical} h^*\rho_{X}^2\approx
dt^2+\sum_{i=1} ^{m-1}\varphi_i ^2 (t;x) \cdot
\lambda_i^2(x),\end{equation} for $(t;x)$ in a dense subset of
$ (0;\ep)\times L(x_0;X_{reg}).$
As the $\varphi_i$'s are bounded below and above far away from $\{t=0\}$
we see that the above mapping $h$ is thus a quasi-isometry  on any closed subset of $X_{reg}$ disjoint from $ \{t=0\}$.
\end{enumerate}
\end{rem}

\medskip

\begin{thm}\label{thm_hardt}
Let $x_0 \in X\subset \R^n$ and set $\rho(x):=|x-x_0|$. There exists $\ep>0$ such that $\rho$ is bi-Lipschitz trivial above
$[\nu;\ep]$ for any $0<\nu <\ep$, i. e. for every $\nu>0$ we can
find a bi-Lipschitz homeomorphism
$$h: \rho^{-1}([\nu;\ep]) \to \rho^{-1}(\ep)\times
[\nu;\ep], $$ with $\pi_2(h(x))=\rho(x)$, where
$\pi_2:\rho^{-1}(\ep)\times [\nu;\ep] \to [\nu;\ep]$ is the
projection onto the second factor.
\end{thm}

This theorem is a particular case of  the bi-Lipschitz version of
Hardt's Theorem proved in \cite{vlt}. This is also easy to derive
from the proof of Theorem \ref{thm Lipschitz conic structure}. The subanalycity
of the isotopy will be useful in section \ref{sect_hom} (recall
that, except the differential forms, everything is implicitly assumed to be subanalytic).

\medskip


\section{$L^1$ cohomology groups}\label{sect_weakly_l1}
 In this section $M\subset \R^n$ stands for a  bounded
 (subanalytic)
submanifold. Such a manifold has a natural structure of Riemannian
manifold giving rise to a measure on $M$ that we denote $dvol_M$.
Below, the word $L^1$ will always mean $L^1$ with respect to this
measure.

\subsection*{The $L^1$ cohomology.} As we said in the
introduction, the {\bf $L^1$ forms on $M$} are the forms $\omega$
on $M$ satisfying (\ref{eq_l1_condition}). We denote by
$(\Omega_{(1)} ^\bullet(M);d)$ the differential complex constituted by the
$C^\infty$ $L^1$ forms $\omega$ such that $d\omega$ is $L^1$.

The $L^1$ cohomology groups, denoted $H_{(1)}^j(M)$ are the
cohomology groups of the differential complex $(\Omega_{(1)}
^\bullet(M);d)$.

We endow this de Rham complex with the natural norm:
$$|\omega|_1:=\int_M |\omega|\, dvol_M+ \int_M |d\omega|\,d vol_M.$$

In this section we prove some preliminary results about $L^1$
cohomology that we shall need to establish our de Rham theorem in
the next section.

%

\subsection{$L^1$ cohomology with compact support.}\label{sect_cohomo_support_compact} We now define the $L^1$ forms with compact support. We prove some basic facts, relying on the bi-Lipschitz
triviality result presented in Theorem \ref{thm_hardt}. Let us
point out that our notion of forms with compact support is
slightly different that the usual one since we  allow the forms to
be nonzero near the singularities of $cl(M)$. The support is
indeed a subset of $cl(M)$.

Let $M\subset \R^n$ be a submanifold and let $X:=cl(M)$.

\begin{defs}\label{dfn_support}
Let $U$ be an open subset of $M$ and let $V\supset U$ be an open subset of $X$. Let $\omega$ be a differential form on $U$.
   The {\bf support of $\omega$ in $V$} is the closure in $V$ of the set constituted by the  points of $U$  at which $\omega$ is nonzero.

  We denote by $\Omega^j_{(1),V}(U)$
the  $C^\infty$  $j$-forms $\omega$ on $U$ with  compact
support in $V$  such that $\omega$ and $d\omega$ are $L^1$,  and
by $H^j_{(1),V}(U)$ the resulting cohomology groups.
\end{defs}

For instance $\Omega^j_{(1),X}(M)$ stands for the $L^1$ $j$-forms (with an $L^1$ derivative) having compact support in $X$.
Such forms have to be zero in a neighborhood of infinity (in $M$). However, they need not to be zero
near the points of $\delta M$.

\subsection{Weakly differentiable forms.} The homeomorphism that we constructed in Theorem \ref{thm Lipschitz conic
structure} is not smooth.  Thus, we will need to work with weakly differentiable forms, 
just differentiable as currents. Therefore, the first step is to
prove that the bounded  weakly differentiable forms give rise to
the same cohomology theory. We will follow an argument similar to
the one used by Youssin in \cite{y}.

Given    a smooth manifold $M$ (possibly with boundary), we denote
by $\Omega_{0,\infty} ^{j}(M)$ the set of $C^\infty$ $j$-forms on $M$
with compact support (in $M$).

\begin{dfn}
Let $U$ be an open subset of $\R^n$. A differential
$j$-form $\alpha$ on $U$ is called {\bf weakly differentiable} if
there exists a $(j+1)$-form $\omega$ such that for any form
$\varphi\in \Omega_{0,\infty} ^{n-j-1} (U)$:$$\int_{U} \alpha \wedge
d\varphi =(-1)^{j+1}\int_{U} \omega \wedge \varphi.$$ The form
$\omega$ is then called {\bf the weak exterior derivative of
$\alpha$} and we write $\omega=\overline{d} \alpha$.   A
continuous differential $j$-form $\alpha$ on $M$ is called {\bf
weakly differentiable} if it  gives rise to   weakly
differentiable forms via the coordinate systems  of $M$.
\end{dfn}

We denote by $\overline{\Omega}_{(1)} ^{j}(M)$  the set of
measurable weakly differentiable $j$-forms, locally bounded in $M$, which are  $L^1$ and
which have an  $L^1$ weak exterior derivative. Together with
$\overline{d}$, they constitute  a cochain complex. We denote by
$\overline{H}_{(1)}^j (M)$  the resulting cohomology groups.

We endow this de Rham complex with the corresponding  norm:
$$|\omega|_1:=\int_M |\omega|\, dvol_M+ \int_M |\overline{d}\omega|\,d vol_M.$$

Similarly, we may introduce the theory of {\bf weakly
differentiable $L^1$ forms with compact support in $V$} that we shall
denote $\overline{\Omega}^j_{(1),V}(U)$ and
$\overline{H}^j_{(1),V}(U)$ (see definition \ref{dfn_support}).

In the case of compact smooth manifolds it is easily checked that
the two cohomology theories coincide:

\begin{lem}\label{lem_l1_manifold}If $K$ is a smooth  compact manifold (possibly with boundary) then:
 \begin{equation}\label{eq isom bar}\overline{H}_{(1)} ^j(K) \simeq
H^j(K).\end{equation}\end{lem}
\begin{proof}
The proof follows the classical argument. As in the  case of
smooth forms (see for instance \cite{bl}) it is enough to show
Poincar\'e Lemma. Both of the above cohomology theories  are
invariant under smooth homotopies. Any point of $K$ has a smoothly
contractible neighborhood. As $K$ is compact, locally $L^1$
implies $L^1$.
\end{proof}

We now are going to see that the isomorphism also holds in the noncompact
case:

\begin{pro}\label{pro l1 isom smooth}
Let   $M\subset \R^n$ be a $C^\infty$  submanifold and let $V$ open in $cl(M)$.   The inclusions $\Omega_{(1)}
^\bullet (M) \hookrightarrow \overline{\Omega}_{(1)} ^\bullet (M)$ and $\Omega_{(1),V}
^\bullet (V\cap M) \hookrightarrow \overline{\Omega}_{(1),V} ^\bullet (V \cap M)$ induce
isomorphisms between the cohomology groups.
\end{pro}
\begin{proof} As the proof is the same for the two inclusions, we shall focus on the former one.
It is enough to show that, for any form $\alpha \in
\overline{\Omega}_{(1)} ^{j} (M)$   with $\overline{d}\alpha \in
\Omega_{(1)} ^{j+1} (M)$ (i. e. $\alpha$ is weakly smooth and
$\overline{d}\alpha$ is smooth),
  there exists  $\theta \in \overline{\Omega}_{(1)} ^{j-1} (M)$ such that $(\alpha+\overline{d}\theta)$ is
  $C^\infty$. For this purpose, we prove by induction on $i$ the
  following statements.

{\bf$(\textrm{H}_i)$} Fix a form $\alpha \in
\overline{\Omega}_{(1)} ^{j} (M)$  with $\overline{d}\alpha \in
\Omega_{(1)} ^{j+1} (M)$. Consider an exhaustive sequence of
compact smooth manifolds with boundary $K_i \subset M$ such that
for each
 $i$, $K_i$ is included in the interior of
$K_{i+1}$ and $\cup K_i=M$.
  Then, for any integer $i$, there exists a closed  form
  $\theta_i \in \overline{\Omega}_{(1)} ^{j-1}(M)$ such that
  $supp \; \theta_i \subset Int(K_{i}) \setminus K_{i-2}$ and $|\theta_i|_1\leq
  \frac{1}{2^i}$ and such that $\alpha_i:=\alpha + \sum_{k=1} ^i \overline{d}\theta_k$ is smooth in a neighborhood of $K_{i-1}$.

  Before proving these statements observe that $\theta=\sum_{i=1}
  ^\infty  \theta_i$ is the desired exact form (this sum is locally
  finite).

Let us assume that $\theta_{i-1}$ has been constructed, $i \geq 1$
(we may set $K_0=K_{-1}=K_{-2}=\emptyset$). Observe that by
(\ref{eq isom bar}), there exists a smooth form $\beta \in
\Omega^{(j-1)}_{(1)} (K_i)$ such that $d\beta=\overline{d}\alpha$.
This means that $(\alpha_{i-1}-\beta)$
   is $\overline{d}$-closed, and by  (\ref{eq isom bar}) there is a smooth form $\beta' \in  \Omega^{(j-1)}_{(1)} (K_i)$ such that
   $$\alpha_{i-1}-\beta=\beta' +\overline{d} \gamma,$$ with $\gamma \in
   \overline{\Omega}^{(j-2)}_{(1)} (K_i)$ (if $j=1$ then $\alpha_{i-1}-\beta$ is constant and then
   smooth).  Thanks to the induction hypothesis
     there exists an open  neighborhood $V$ of $K_{i-2}$ on which
     $\alpha_{i-1}$ is smooth. This implies that $\overline{d} \gamma$ is smooth on $V$.
      Therefore, by induction, we know that we can add an exact
   form  $d\sigma$ to $\gamma$ to get a form smooth on $V$. Multiplying $\sigma$ by a function with support in $V$
    which is $1$  in a neighborhood $W$ of $K_{i-2}$, we get a form $\sigma'$ on $M$ such that $(\overline{d}\sigma'+\gamma)$ is smooth on $W$.
      This means that
   we can assume that $\gamma$ is smooth on an open neighborhood $W$ of $K_{i-2}$ possibly replacing $\gamma$ by $(\overline{d}\sigma'+\gamma)$.
    We will assume
   this fact without changing notations.

    By means of a convolution product with bump functions, for any
    $\ep>0$,
     we may construct a smooth form  $\gamma_\varepsilon$ such that $|\gamma_\varepsilon-\gamma|_{1}\leq \varepsilon$.

     Consider a smooth function $\phi$ which is $1$ on a neighborhood of
     $(M \setminus W)\cap K_{i-1}$ and with support in $int(K_i) \setminus
     K_{i-2}$. Then set:
$$\theta_i(x):=\phi(x)(\gamma_\ep-\gamma)(x).$$

 If
$\varepsilon$ is chosen small enough $|\theta_i|_1
+|d\theta_i|_1\leq \frac{1}{2i}$. On a neighborhood of
     $(M \setminus W)\cap K_{i-1}$, because $\phi\equiv 1$,
we have
$\alpha_{i-1}+\overline{d}\theta_i=\beta+\beta'+d\gamma_\varepsilon$
which is smooth. The form $(\alpha_{i-1}+\overline{d}\theta_i)$ is
smooth on $W$ as well since
     $\alpha_{i-1}$ and $\theta_i$ are both smooth.
      \end{proof}

%
%

\subsection{Weakly smooth forms and bi-Lipschitz maps.}
Given two open subsets of $\R^n$, it is well known that  any
subanalytic map   $h:U \to V$ is smooth almost everywhere.
Therefore, any form $\omega$ on $V$ may be pulled-back to a form
$h^*\omega$ on $U$, defined  almost everywhere.


We are going to see that in the case where $h$ is locally
bi-Lipschitz then the pull-back of a smooth form is weakly smooth
(Proposition \ref{prop_pullback_weakly smooth forms}).

\begin{dfn}
Let $\Sigma$ be a   stratification of $U\subset \R^k$ and let $h:U
\to \R^n$  be smooth on strata. The map $h$ is {\bf horizontally
$C^1$ (with respect to $\Sigma$)} if, for any sequence
$(x_l)_{l\in \N}$ in a stratum $S$ of $\Sigma$ tending to some
point $x$ in a stratum $S'$ and for any  sequence $u_l \in
T_{x_l} S$ tending to a vector $u$ in $T_x S'$, we have
$$\lim d_{x_l} h_{|S}(u_l)=d_x h_{|S'} (u).$$
\end{dfn}

Horizontally $C^1$ maps have been introduced by David Trotman and
Claudio Murolo in \cite{mt}. They will be useful to show that the
pull-back of a weakly differentiable $L^1$ form by a subanalytic
bi-Lipschitz map (not everywhere smooth) is weakly differentiable.

 The
following lemma will be needed. Similar results were proved in
\cite{sv} where the theory of stratified forms is investigated and
a de Rham type theorem for these forms is proved.

\begin{lem}\label{lem_h_hor_C1}
Let $h:U\to \R^m$ be a Lipschitz map. There
exists a stratification of $U$ such that $h$ is horizontally $C^1$
with respect to this stratification.
\end{lem}
\begin{proof}Consider a Whitney $(a)$ stratification $\Sigma_h$ of the graph of
$h$ (see for instance \cite{bcr, tesc} for the definition of
the Whitney $(a)$ condition and the construction of such a
stratification). Let $\pi_1$ (resp. $\pi_2$) be the projection on
the source (resp. target) axis of $h$. The image of $\Sigma_h$
under $\pi_1$ gives rise to a stratification $\Sigma$ of $U$. Let
us prove that $h$ is horizontally $C^1$ with respect to this
stratification. Fix a stratum $S$ of this stratification, a
sequence $x_l \in S$ tending to $x$ belonging to a stratum $S'$, as well as
a sequence $u_l \in T_{x_l} S$ of vectors tending to some $u \in T_x S'$.
Let $Z$ be the stratum which projects onto $S$ via $\pi_1$. For every $l$, there
is a unique vector $v_l \in T_{(x_l;h(x_l))} Z$ which projects
onto $u_l$. As $h$ is Lipschitz the norm of $v_l$ is bounded above
and we may assume that $v_l$ is converging to a vector $v$. The
vector $v$ then necessarily projects onto $u$.

We claim that $v$ is tangent to the stratum $Z'$ of $\Sigma_h$
containing $(x;h(x))$. Indeed, if otherwise, there would be a
vector $w$ in $\tau=\lim T_{(x_l;h(x_l))} Z$ such that $(w-v)$
lies in the kernel of $\pi_1$, in contradiction with the fact that
$h$ is Lipschitz (the graph of Lipschitz map may no have a
vertical limit tangent vector). This shows the claim, and
consequently:
$$\lim d_{x_l} h_{|S} (u_l)=\lim \pi_2(v_l)=\pi_2(v)= d_{x} h_{|S'}(u),$$
since $v$ is tangent to $Z'$.
\end{proof}

We shall need the following fact on subanalytic homeomorphisms. It seems that it could be improved but this will be enough for our purpose. 

\begin{pro}\label{prop_pullback_weakly smooth forms}
Let $U$ be an open subset of $\R^n$ and let $\omega$ be a bounded
weakly differentiable form on $U$ with  $\overline{d}\omega$
bounded. If $h:U\to V$  is a locally
bi-Lipschitz map, then $h^*\omega$ is weakly differentiable and
$\overline{d} h^*\omega=h^* \overline{d} \omega$, almost
everywhere.
\end{pro}
\begin{proof}Take $\varphi \in
\Omega_{0,\infty} ^{m-j}  (U)$.

 \underline{{\it First case:}} assume that
$\omega$ is smooth.  Let $\rho$ be the function defined by the
distance to the boundary of $U$ and set $U^\ep := \{\rho \geq
\ep\} $.

By Lemma \ref{lem_h_hor_C1}, $h$ is horizontally $C^1$ with
respect to some stratification of $U$. Consequently, the forms $
h^*\omega$ and $h^*d\omega$ are  continuous at almost every point of
$cl(U^\ep)$ (it is a manifold with boundary a. e.). Hence, so are
$h^*\omega \wedge \varphi$ and $h^*d\omega \wedge \varphi$. The form $h^*\omega$ is smooth almost everywhere. By Stokes'
Formula for stratified forms \cite{sv} (see also \cite{l}),
 $$ \int_{U^\ep}d( h^*\omega \wedge \varphi)=  \int_{\rho =\ep} h^*\omega \wedge \varphi=0,$$
 for $\ep>0$ small enough, since $\varphi$ has compact support in $U$.

 Now, integrating by parts we have for $\ep>0$ small enough:
 $$(-1)^{j+1}\int_{U} h^*\omega \wedge d \varphi = \int_U d h^*\omega \wedge \varphi - \int_{U}
 d( h^*\omega \wedge \varphi)= \int_U d h^*\omega \wedge \varphi. $$
This completes the proof of our first case.

In general, if $\omega$ is not smooth but just weakly smooth, as
$\varphi$ is smooth and $h^{-1}$ bi-Lipschitz, $ h^{-1*} d
\varphi$ is weakly smooth (by the {\it First case} applied to $\varphi$ and $h^{-1}$) and  we may
write:
 $$\int_{U}  h^*\omega \wedge d \varphi
  =\int_V \omega \wedge h^{-1*} d \varphi = \int_V \omega \wedge \overline{d}  h^{-1*}  \varphi ,$$
and, again integrating by parts:
$$  \int_V \omega \wedge \overline{d}  h^{-1*}  \varphi  =(-1)^{j+1}\int_V \overline{d} \omega \wedge h^{-1*}  \varphi=(-1)^{j+1}\int_{U} h^* \overline{d} \omega \wedge  \varphi.$$
\end{proof}

\subsection{Subanalytic bi-Lipschitz maps and $L^1$ cohomology.}\label{sect_hom} In
general, if $f:M\to N$ is a weakly smooth map between smooth
manifolds and if $\omega$ is a $L^1$ form on $M$ then $f^*\omega$
is not necessarily a $L^1$ form on $N$, even if $f$ has bounded
derivatives. Nevertheless, if $f$
is a diffeomorphism and if $|d_x f^{-1}|$ is bounded above then the pullback of a $L^1$ form is $L^1$. 

In particular, if $f$ is a subanalytic bi-Lipschitz map, by
Proposition \ref{prop_pullback_weakly smooth forms}, $h^*\omega$
is a weakly smooth $L^1$ form (it is well defined almost
everywhere). This means that any subanalytic bi-Lipschitz map
$h:M\to N$ induces some maps
$$h^{*\bullet}:\overline{\Omega}_{(1)}^\bullet(N)
\to\overline{\Omega}_{(1)}^\bullet(M),$$  pulling-back the forms.
These mappings induce mappings in cohomology which are  obviously
isomorphisms since $h$ is invertible.

Fix a $C^\infty$ submanifold $M\subset \R^n$. Let $x_0 \in cl(M)$, and set $M^\ep :=B^n(x_0;\ep)$ as well as
$N^\ep:=M\cap S^{n-1}(x_0;\ep)$.

\begin{pro}\label{pro_neigh}
 For any $\ep$ positive small enough,
there exists a fundamental system of neighborhoods
$(U_i)_{i\in \N}$ of $N^\ep$  such that:
$$H_{(1)} ^\bullet (U_i\cap M^\ep) \simeq H_{(1)} ^\bullet (L(x_0;M)).$$
\end{pro}
\begin{proof}
 By Proposition \ref{pro l1 isom
smooth}, it is enough to show  the result for the $L^1$ cohomology
of weakly smooth forms.  Apply Theorem
\ref{thm_hardt} to $cl(M)$. Then set
$$U_i:=\rho^{-1}((\ep-\frac{\ep}{2i};2\ep)),$$ for $i$ positive
integer (with the notations of the latter theorem). Now the
 bi-Lipschitz homeomorphism provided by Theorem
\ref{thm_hardt} induces an isomorphism (as explained in the
paragraph preceding the proposition) between $\overline{H}_{(1)}^j(U_i\cap M^\ep)$ and $\overline{H}_{(1)}
^j(N^{\nu} \times (\ep-\frac{\ep}{2i};\ep))$, for any $\nu \in (\ep-\frac{\ep}{2i};\ep)$. It is a routine to
check that the latter is isomorphic to $\overline{H}_{(1)} ^j(N^\nu)$.
\end{proof}

\begin{rem}\label{rmk_restriction_map}
We recall that the link is defined as the intersection of the set with a little sphere, say that it is $ N^\nu$.  In the above proposition, the isomorphism is induced by
restriction. Of course, the restriction of a $L^1$ form on $M^\ep$
has no reason to give rise to a  $L^1$ form on $N^{\nu}$ but every class has a representative which is $L^1$ in
restriction to $N^\nu$, since the isomorphism
$$\overline{H}_{(1)}^j(N^{\nu})\simeq \overline{H}^j_{(1)}(N^{\nu} \times (\ep-\frac{\ep}{2i};\ep))$$ involved in the above proof
is itself induced by the restriction.
\end{rem}

\subsection{An exact sequence nearby singularities.}\label{sect_cohomo_cpct_supp_local} The letter $M\subset \R^n$
still stands for a  $C^\infty$ submanifold. We shall point out an
exact sequence nearby a singular point of the closure of  $M$. Fix
$x_0\in X$ and set $M^\ep:=B^n(x_0;\ep)\cap M$, $N^\ep:=S^{n-1}(x_0;\ep)\cap
N$ as well as $X^\ep:=B^n(x_0;\ep)\cap X$.

 By
Proposition \ref{pro_neigh},   for any $\ep$ small enough, there
is a basis of neighborhoods $(U_i)_{i \in \N}$ of $N^\ep$ for
which the restriction map (see remark \ref{rmk_restriction_map})
induces an isomorphism for every $i$:
\begin{equation}\label{eq_U_i} H_{(1)}^j(U_i)\simeq
H_{(1)}^j(N^\ep).\end{equation}

Denote by $\hat{\Omega}_{(1)}^j (N^\ep)$ the direct limit of
$\Omega_{(1)}^j (U\cap M)$  where $U$ runs over all the neighoborhoods of $N^{\ep}$.  Denote by  $\hat{H}_{(1)}^j(N^\ep)$ the
resulting cohomology (these groups are indeed isomorphic to $H_{(1)}^j (N^\ep)$ thanks to Proposition \ref{pro_neigh}). The short exact sequences
$$0 \to \Omega_{(1),X^\ep} ^\bullet (M^\ep)\to \Omega_{(1)} ^\bullet (M^\ep)\to   \hat{\Omega}_{(1)}^\bullet (N^\ep)\to  0, $$
give rise to the following long exact sequence:
$$\dots \to \hat{H}_{(1)}^{j-1} (N^\ep) \to H_{(1),X^\ep} ^j (M^\ep) \to  H_{(1)} ^j (M^\ep)\to   \dots .$$

Similarly let $C^\bullet _{X^\ep} (M^\ep)$ be the singular cohomology
with compact support in $X^\ep$, i. e. the singular cochains of $M^\ep$ whose support does not meet any neighborhood of $S^{n-1}(x_0;\ep)$. Consider now the mappings:
$$\psi_{M^\ep,X^\ep }^\bullet: \Omega^\bullet _{(1),X^\ep}(M^\ep) \to C^\bullet_{X^\ep} (M^\ep),$$
obtained in the same way as $\psi_{X^\ep} ^\bullet$, by
integrating the $L^1$ differential forms on simplices.

 The above exact sequence, together with the analogous exact
sequence in singular cohomology, provide the following commutative
diagram: \vskip 0.5cm

\begin{center}
     \begin{picture}(-140,0)
\put(103,-95){$\mbox{diag.} \, 1. $}
 \put(-240,0){$\dots\longrightarrow H_{(1),X^\ep}^j (M^\ep)\; \longrightarrow \; H^{j}_{(1)}(M^\ep)
 \longrightarrow\;
            H^{j}_{(1)}(N^\ep) \; \longrightarrow \; H^{j+1}_{(1),X^\ep}(M^\ep) \; \longrightarrow \dots$}
 \put(-255,-70){$\;\dots \; \longrightarrow H_{X^\ep} ^j (M^\ep) \;\;\longrightarrow \;\; H^{j}
 (M^\ep)\;\:\longrightarrow \:\;
            H^{j} (N^\ep) \;\; \longrightarrow \; H^{j+1} _{X^\ep} (M^\ep) \longrightarrow \dots$}
\put(-190,-9){\vector(0,-1){45}} \put(-185,-30){$\psi^j
_{M^\ep,X^\ep}$}
 \put(-110,-9){\vector(0,-1){45}}\put(-105,-30){$\psi^j_{M^\ep}$}
 \put(-40,-9){\vector(0,-1){45}}\put(-35,-30){$\psi^{j}_{N^\ep}$}
\put(25,-9){\vector(0,-1){45}}
\put(30,-30){$\psi^{j+1}_{M^\ep,X^\ep}$}
  \end{picture}
    \end{center}
\vskip 4cm

\section{Proof of the de Rham theorem for $L^1$
cohomology.}\label{sect_de_rham_l1}
 The first step of the proof of  Theorem \ref{thm_intro} is to
compute the cohomology groups locally. This requires to construct
some homotopy operators and describe their properties.

 The letter  $M\subset \R^n$ still stands for a bounded  submanifold. Set $X:=cl(M)$ and take
$x_0\in X$. Set again for simplicity $M^{\ep}:=M \cap B^n(x_0;\ep)$ and
$N^\ep:=M \cap S^{n-1}(x_0;\ep)$  as well as $X^\ep:=B^n(x_0;\ep)\cap X$ (we do not match $x_0$ since it is arbitrary).

\subsection{Some operators on weakly smooth forms.}\label{sect_hom_hop}
For $\ep >0$ small enough and $j>0$ fixed,   we are going to construct operators for weakly smooth forms. 

For this purpose, apply Theorem \ref{thm Lipschitz conic
structure} to $X$ at $x_0$. Let $h: (0;\ep)\times N^\ep\to M^\ep$ be the homeomorphism described
in definition \ref{dfn conical} and
  fix $\omega \in  \overline{\Omega}_{(1)} ^j (M^\ep)$ with $j\geq 0$ (where $\ep$ is also provided by definition \ref{dfn conical}).
 Set now  $Z:= (0;\ep)\times N^\ep$. 

 We may define two forms  $\omega_1$ and $\omega_2 \in
\overline{\Omega}_{(1)}^j(Z)$ by:
$$h^*\omega(t;x):=\omega_1(t;x)+dt\wedge \omega_2(t;x),$$
where $\omega_1$ and $\omega_2$ do not involve the differential term $dt$. 
The forms $\omega_1$ and $\omega_2$ are indeed only defined for almost
every $(t;x) \in Z$. 

Next, we set for almost every $(t;x) \in Z$ and $0<\nu\leq \ep$:
 \begin{equation}\label{eq_def_alpha}\alpha(t;x):=\int_\nu ^t  \omega_2 (s;x) ds,\end{equation} and
 \begin{equation}\label{eq_def_K}
  \K_\nu \omega:=h^{-1*}\alpha.
 \end{equation}

 We first show that $\K_\nu$ preserves the weakly smooth forms.

\bigskip

\subsection*{The mapping $\pi_\nu$.} Given $\omega \in \overline{\Omega}_{(1)} ^j (M^\ep)$ and  $\nu\leq \ep$, let   $\pi_\nu:=h\circ P_\nu \circ h^{-1}$, where $P_\nu (t;x):=(\nu ;x)$.  Given a differential form $\omega $ on $M^\ep$ we will denote by $\pi_\nu ^*\omega$ the form given by the pull-back of $\omega$ by means of $\pi_\nu:M^\ep \to M^\ep$.  

\bigskip

\begin{lem}\label{lem_proprietes_de_K_nu}
 For $M$  as above, $\K_\nu$ preserves the  weakly smooth forms and satisfies on  $\overline{\Omega}_{(1)} ^j (M^\ep)$:
 \begin{equation}\label{eq_K_nu_homot_operator}\overline{d}\K_\nu-\K_\nu \overline{d}=Id-\pi_\nu^*,\end{equation}
\end{lem}
\begin{proof}
Take $\omega$ in $  \overline{\Omega}_{(1)} ^j (M^\ep)$ and  
let us fix a  form $\varphi \in \Omega^{m-j} _{0,\infty} (M)$. Let $h$ be as above.

 The mapping
$h$  is locally bi-Lipschitz in $h^{-1}(M^\ep)$ (see Remark \ref{rem_conical} (\ref{rem_conical_item3})).
   By Proposition
\ref{prop_pullback_weakly smooth forms},  the form $h^*\omega$ is
weakly differentiable and $\overline{d} h^* \omega=
h^*\overline{d}\omega$ and the same is true for $\varphi$. Let $\alpha$ be the form defined in (\ref{eq_def_alpha}) and set $\psi:=h^*\varphi$.  It is enough to show:
$$(-1)^{j}\int_Z \alpha \wedge \overline{d}\psi = \int_Z h^*\K_\nu d\omega \wedge \psi+\int_Z h^*\omega \wedge \psi- \int_Z h^*\pi_\nu ^* \omega \wedge \psi . $$

For this purpose, note that we have (for relevant orientations):
\begin{eqnarray*}
(-1)^{j}\int_{Z} \alpha \wedge \overline{d}\psi&=&(-1)^{j}\int_0 ^\ep(\int_{ [\nu;t]\times N^\ep
} h^*\omega \wedge \overline{d}\psi)\,dt
\\&=&  \int_0 ^\ep\int_{ [\nu;t]\times
N^\ep} \overline{d}h^*\omega \wedge\psi-\int_0 ^\ep\int_{[\nu;t]\times
N^\ep}\overline{d}(
h^*\omega(s;x) \wedge \psi(t;x) ) ,
\end{eqnarray*}
(integrating by parts) and therefore if $\Delta_\nu:=\{(s;t): \nu \leq s \leq t < \ep  \; \mbox{or} \; 0<  t \leq s \leq \nu\}$ we have:
$$(-1)^{j}\int_{Z} \alpha \wedge \overline{d}\psi =\int_{Z}
h^*\K_\nu d \omega \wedge \psi- \int_{N^\ep}\int_{\Delta_\nu}\overline{d}( h^*\omega(s;x) \wedge \psi(t;x) ) . $$

But, since  $\psi$ has compact
support in $M^\ep$,  by Stokes' formula  we have:
$$\int_{N^\ep}\int_{\Delta_\nu}\overline{d}(
h^*\omega(s;x) \wedge \psi(t;x) )=\int_Z h^*\pi_\nu ^* \omega \wedge \psi -\int_{Z} h^*\omega\wedge \psi.$$
Together with the preceding equality this implies that
$$(-1)^{j}\int_{Z} \alpha \wedge \overline{d}\psi =\int_{Z} h^*\omega\wedge
\psi +\int_{Z} h^*\K_\nu d \omega \wedge \psi - \int_Z h^*\pi_\nu ^* \omega \wedge \psi ,$$ as required. 
\end{proof}

\subsection*{The homotopy  operator $\K$.} We derive from $\K_\nu$ a local homotopy operator $\K$. 

Let $\ep >0$ be as above  and let $j>0$. We  just saw that $\K_\nu$ preserves the weakly smooth forms. Observe that if $\omega$ has compact support in $X^\ep$ then $h^*\omega(\nu;x)$ is zero for $\nu<\ep$ sufficiently close to $\ep$. Therefore 
$\K_\nu$ induces an  operator:$$\K: \overline{\Omega}_{(1),X^\ep} ^j (M^\ep)\to
\overline{\Omega}_{(1),X^\ep} ^{j-1} (M^\ep),$$
defined by the stationary limit $\K\omega:=\lim_{\nu \to \ep} \K_\nu$.
\bigskip
Below we describe the properties of $\K$. 

\begin{pro}\label{pro_proprietes_de_K}
 For $M$  as above,  $\K$ is a homotopy operator, in the sense that:
\begin{equation}\label{eq_K_homot_operator}\overline{d}\K-\K \overline{d}=Id,\end{equation}
bounded for the $L^1$ norm and satisfying for $j<m$:
\begin{equation}\label{eq_int_komega_link}
 \lim_{t\to 0} \int_{N^{t}} |\K\omega|=0,
\end{equation}
for any $\omega \in \overline{\Omega}_{(1),X^\ep} ^j (M^\ep)$.
\end{pro}
\begin{proof}
As observed in the paragraph preceding the proposition, if $\omega$ has compact support in $X^\ep$ then $h^*\omega(\nu;x)$ vanishes near $\nu =\ep$, and thus $\pi_\nu ^*\omega$ is zero if $\nu$ is sufficiently close to $\ep$.  As a matter of fact, equality (\ref{eq_K_homot_operator})   follows from (\ref{eq_K_nu_homot_operator}).  We have to check that $\K$ is bounded for the $L^1$ norm and show (\ref{eq_int_komega_link}). 

{\bf Some notations.} We shall write $\I_k^{m}$ for the set  all the
multi-indices  $I=(i_1,\dots,i_k)$ with $0< i_1<\dots
 <i_k<m$. Given $I \in \I_k$ we shall write $\hat{I}$ for the
 multi-index of $\I_{m-1-k}$
  such that $I \cup \hat{I}=\{1,\dots,m-1\}$.
Let $\lambda_1,\dots,\lambda_{m-1}$ be the tame basis of $1$-forms
provided by definition \ref{dfn conical} (on a dense subset $N'$
of $N^\ep$) and set for any multi-index
 $\lambda_I:=\lambda_{i_1}\wedge \dots \wedge \lambda_{i_k}$.

We now are going to show that the operator $\K$ is bounded for the
$L^1$ norm.

As $(\lambda_1;\dots;\lambda_{m-1})$ is a tame basis of $N^\ep$ we may decompose
     $\alpha:=\sum_{I \in \I_{j-1}} \alpha_I \lambda_I$ (where $\alpha$ is the form defined in (\ref{eq_def_alpha}))
   and observe that by $(3)$ of definition \ref{dfn conical} \begin{equation}\label{eq_K_omega_simeq} |\K \omega|= |\sum_{I \in \I_{j-1}}h^{-1*} \alpha_I |\sim \sum_{I \in \I_{j-1}} \frac{|\alpha_I\circ h^{-1}|}{\varphi_I\circ h^{-1}},\end{equation} where
 $\varphi_I =\varphi_{i_1} \cdots \varphi_{i_k}$,
 and consequently it is enough to show that all the summands of the
right hand side are $L^1$ on $M^\ep$. Changing variables by means of $h$, this amounts to show that for any
$I \in \I_k$:
$$\int_{Z} |\alpha_I|\cdot\frac{J_h}{\varphi_I}  <\infty,$$
where $J_h$ stands for the absolute value of the Jacobian
determinant of $h$.

  Alike, decompose  $$\omega_{2}=\sum_{I \in \I_{j-1}} \omega_{2,I} \lambda_I,$$ (recalled that we decomposed $h^*\omega:=\omega_1+dt \wedge \omega_2$).
     As  $(\lambda_1;\dots;\lambda_{m-1})$ is a tame basis of $N^\ep$  we have$|\omega_2|\sim \sum_{I\in \I_{j-1}} |\omega_{2,I}|$.
 For the same reasons as in (\ref{eq_K_omega_simeq}):
\begin{equation}\label{eq maj de omega}|\omega_{2,I}(s;x)| \leq C |\omega (h(s;x))| \cdot \varphi_I(s;x). \end{equation}

By  $(3)$ of definition \ref{dfn conical}   we have
on $Z$:
\begin{equation}\label{eq_varphi_et_jh}\varphi_I \cdot
\varphi_{\hat{I}}\sim J_h\end{equation}


Put $Y^t :=\{t\} \times N^\ep$. There is a constant $C$ such that  for almost every $t$ and any $I\in
\I_k$:
\begin{eqnarray}
\int_{Y^t} \frac{|\alpha_I|}{\varphi_I} \cdot J_h &\leq &C \int_{x \in N^\ep } |\alpha_I(t;x)|\cdot\dfrac{J_h(t;x)}{\varphi_{I}(t;x)} \nonumber\\
&\leq & C\int_{N^\ep } |\alpha_I(t;x)|\cdot \varphi_{\hat{I}}(t;x) \quad \mbox{(by (\ref{eq_varphi_et_jh}))}\nonumber\\
&\leq & C\int_{N^\ep } \int_t ^\ep |\omega_{2,I} (s;x)|\cdot \varphi_{\hat{I}}(t;x)ds\quad \mbox{(by (\ref{eq_def_alpha}))}\nonumber \\
&\leq & C\int_{N^\ep} \int_t ^\ep |\omega (s;x)|\cdot \varphi_{I}(s;x) \cdot \varphi_{\hat{I}}(t;x)ds \quad \mbox{(by (\ref{eq maj de omega}))}\label{eq_calcul_L1norm_K}\\
&\leq & C\int_{N^\ep} \int_t ^\ep |\omega(s;x)|\cdot \varphi_{I}(s;x)
\cdot \varphi_{\hat{I}}(s;x)ds\nonumber 
\end{eqnarray}
since, by $(1)$ of definition \ref{dfn conical},
$\varphi_{\hat{I}}(t;x)$ is  nondecreasing with respect to $t$. We finally
get:
\begin{equation}\label{eq_calcul_L1norm_KIII}\int_{Y^t} \frac{|\alpha_I|}{\varphi_I} \cdot J_h   \leq C\int_{ (t;\ep)\times N^\ep}  |\omega (s;x)|\cdot J_h(s;x) =  C \int_{h((t;\ep) \times N^\ep ) }  |\omega |.\end{equation}
which is bounded above uniformly in $t$ since $\omega$ is a $L^1$
form,  proving that $\frac{\alpha_I}{\varphi_I}$ is integrable. It remains to establish
(\ref{eq_int_komega_link}). For simplicity set $$f_t(x)=\int_0 ^\ep
|\omega (s;x)|\cdot \varphi_{I}(s;x) \cdot
\varphi_{\hat{I}}(t;x)ds.$$

As $\varphi_{\hat{I}}(t;x)$ is  nondecreasing with respect to $t$, this family of  functions is obviously bounded by the $L^1$ function $ \int_0 ^\ep |\omega(x;s)|\cdot \varphi_{I}(s;x)
\cdot \varphi_{\hat{I}}(s;x)ds$.

Moreover,
as $\varphi_{\hat{I}}$  goes to zero as $t$ tends to zero (since $j<m$), we see
that the function $f_t$ tends to zero (pointwise) as $t$ goes to zero (by the Lebesgue dominated convergence theorem).
Hence, (applying a second time  the Lebesgue dominated convergence theorem) we conclude:
$$ \lim_{t \to 0} \int_{N^\ep} f_t=0. $$

By (\ref{eq_calcul_L1norm_K}):
$$\lim_{t \to 0} \,\int_{Y^t} \frac{|\alpha_I|}{\varphi_I} \cdot J_h  \leq  C \lim_{t \to 0} \int_{N^\ep } f_t=0. $$
But, by (\ref{eq_K_omega_simeq}) this establishes (\ref{eq_int_komega_link}).
\end{proof}

\begin{rem}\label{rem_K}
   Notice  that equation (\ref{eq_calcul_L1norm_KIII}) yields that thee is a constant $C$ such that:
$$\int_{N^t} |\K \omega| \leq C  |\omega|_1,  $$
for any  $t\leq \ep $ and any form $\omega$ in $\Omega_{(1)}^j(M^\ep)$. 
\end{rem}

\bigskip

%

\subsection{Local computations of the $L^1$ cohomology.} The
following proposition may be considered as a  "Poincar\'e Lemma
for $L^1$ cohomology". This is an important step in the proof of
Theorem \ref{thm_intro}.

\begin{pro}\label{Poincare_lem_l1_simple}
 For $\ep>0$ small enough we have for every $j$:
$$H_{(1),X^\ep}^j ( M^\ep ) \simeq 0.
$$
\end{pro}
\begin{proof}
 For $j=0$, a closed  form with compact support is the zero form and the result is clear.  Fix a closed form $\omega  \in  \Omega_{(1),X^\ep} ^j ( M^\ep)$ with $j>0$.
 Let $\K$ be the homotopy operator constructed in the previous section (see Proposition \ref{pro_proprietes_de_K}).
 As $\omega$ is closed with compact support $\overline{d}\K\omega=\omega$, showing that $\omega$ is $\overline{d}$-exact and thus exact by Proposition \ref{pro l1 isom smooth}.
\end{proof}

\subsection{The mappings $\psi_M ^j$}
As in the case of the classical de Rham theorem (for compact smooth
manifolds), the isomorphism is given by integration on simplices.
Let us define this natural map.  Recall that singular simplices $\sigma:\Delta_j \to M $ are assumed to be  subanalytic mappings. Therefore, see \cite{vlinfty} for details, we may define the following maps:
\begin{eqnarray*}\psi_M^j:\Omega_{(1)}^j(M )&\to&
C^{j}(M )\\
 \omega &\mapsto& [\psi_M ^j(\omega):\sigma \mapsto \int_\sigma
\omega].\end{eqnarray*}

By Stokes' formula for singular simplices  \cite{p,sv,vlinfty}, this is a
cochain map.

\subsection{De Rham  theorem for $L^1$ cohomology.}
We are now ready to prove the following theorem, which clearly
implies Theorem \ref{thm_intro}.

\begin{thm}\label{thm_l1_psi_M_isom}
The above mappings $\psi_M ^j$ induce   isomorphisms between the
respective cohomology groups  for any bounded (subanalytic) manifold
$M$.
\end{thm}

\begin{proof}
We prove the theorem by induction on $m$ ($=\dim M$). For $m=0$
the statement is vacuous.

Define a  complex of  presheaves on $X$ by  $\Omega_{(1)} ^j
(U):=\Omega_{(1)}^j(U\cap M)$, if $U$ is an open subset of $X$ and
denote by $\la^j$ the resulting differential sheaf. This is the
sheaf  on $X$ of {\it locally} $L^1$ forms of $M$ (locally in $X$).   Denote by
$\hn^\bullet (\la^\bullet)$ the derived complex of sheaves, i. e. the complex of sheaves obtained from the presheaves  $H^j (\Omega^\bullet _{(1)}(U))$.

 On the other hand, consider the  complex of presheaves on $X$ defined  by  $S^j(U):=C^j(M\cap U),$
for  $U$ open set of $X$, and  denote by $\St^\bullet$  the associated complex of sheaves.

 As the $\la^j$'s are soft sheaves, they are acyclic and it
 follows from the theory of spectral sequences (see for instance \cite{bredon} IV Theorem  2.2) that, if   the sheaf-mappings  $\psi^j : \hn^j(\la^\bullet) \to \hn^j
(\St^\bullet)$,  induced by the morphisms of complexes of presheaves
$\psi_U ^j:\Omega_{(1)}^j(U)\to C^j(U)$, are all isomorphisms, then the mappings $\psi_M^j$ must induce an isomorphism between the cohomology groups of the respective global sections
 of  $\St^{\bullet}$  and $\la^\bullet$. Global sections of $\la^\bullet$ are $L^1$, since, as  $M$ is bounded, $X$ is compact and then locally $L^1$ amounts to $L^1$.

To see that the mappings   $\psi^j : \hn^j(\la^\bullet) \to \hn^j
(C^\bullet)$  are all local isomorphisms,
 we have to show that for every point
$x_0$ in $X$, the mapping $\psi_{X^\ep}$ is an
isomorphism for any $\ep$ small enough. Indeed, by section
\ref{sect_cohomo_cpct_supp_local}, for any $\ep$ small enough,
we have the following commutative diagram for any $j$:
\begin{center}
     \begin{picture}(-140,0)
      \put(-180,-3){$H_{(1)} ^j (M^{\ep})$}
        \put(-90,-3){$H_{(1)} ^j (N^{\ep})$}
      \put(-180,-52){$H^j (M^\ep)$}
          \put(-90,-52){$H^j  (N^\ep)$}
      \put(-133,-50){\vector(3,0){40}}
      \put(-130,0){\vector(3,0){38}}
  \put(-165,-10){\vector(0,-1){30}}
      \put(-70,-10){\vector(0,-1){30}}
      \put(-190,-25){$\psi_{M^\ep}^j$}
      \put(-60,-25){$\psi_{N^\ep}^j$}
     \end{picture}
    \end{center}
\vskip 1.8cm

By Proposition \ref{Poincare_lem_l1_simple} (see diag $1.$), the horizontal arrows
are isomorphisms  for any $\ep$ small enough.

Observe also that $N^\ep$ is of dimension less than $m$.  By
induction on $m$, $\psi_{N^\ep}^j$ induces an isomorphism on
the cohomology groups and thus, the above commutative diagram
clearly shows that the mapping $\psi_{M^\ep} ^j$ induces
an isomorphism  as well for any $j$.
\end{proof}

\medskip

\section{Poincar\'e duality for $L^1$ cohomology}\label{sect_cor_poinc}
We draw some consequences of Theorem \ref{thm_intro}, stating some
duality results between $L^1$ and $L^\infty$ cohomology. We start
by recalling some results and providing basic definitions. We recall that, except the differential forms, all the sets and mappings are assumed to be (globally)
 subanalytic.

\subsection{Intersection homology.}\label{sect_ih}
 We recall the definition of intersection homology as it was introduced by Goresky and Macpherson \cite{ih1,ih2}.
\begin{defs}\label{del pseudomanifolds}
A subset $X\subset \R^n$ is an  {\bf $m$-dimensional pseudomanifold} if
 $X_{reg}$ is an $m$-dimensional manifold which is  dense in $X$ and  $\dim X_{sing}<m-1$.

A {\bf stratified pseudomanifold} is the data of a pseudomanifold
together with a filtration:
$$X_0\subset \dots \subset X_{m-2}= X_{m-1}\subset X_m=X,$$
such that $X_i\setminus X_{i-1}$ is either empty or a smooth
manifold of dimension $i$.

Throughout this section, the letter $X$ will denote a stratified
pseudomanifold.

A {\bf perversity} is a sequence of integers $p = (p_2, p_3,\dots, p_m)$
such that $p_2 = 0$ and $p_{k+1} = p_k$ or $p_k + 1$. A subspace
$Y \subset X$ is called  {\bf $(p;i)$-allowable} if $\dim Y \cap
X_{m-k} \leq p_k+i-k$.  Define $I^{p}C_i (x)$ as the subgroup of
$C_i(X)$ consisting of those chains $\sigma$ such that $|\sigma|$
is $(p, i)$-allowable and $|\partial \sigma|$ is $(p, i -
1)$-allowable.

 The {\bf $i^{th}$ intersection homology group of perversity $p$}, denoted
$I^{p}H_j (X)$, is the $i^{th}$ homology group of the chain
complex $I^{p}C_\bullet(X).$  The {\bf $i^{th}$ intersection cohomology group of perversity $p$}, denoted
$I^{p}H^j (X)$, is defined as $Hom (I^{p}H_j (X);\R)$.
\end{defs}

\medskip

In \cite{ih1,ih2} Goresky and MacPherson have proved that, if the
stratification is sufficiently nice (i. e. if topological
triviality holds along strata) then  these homology groups are
finitely generated and independent of the stratification. Since such
stratifications exist for subanalytic sets \cite{tesc} we
will admit this fact and shall work without specifying the
stratification.

Furthermore, Goresky and MacPherson also proved that their theory satisfy a
generalized version of Poincar\'e duality. We denote by $t$ {\bf
the maximal perversity}, i. e. $t=(0;1;\dots;m-2)$.
\medskip

\begin{thm}(Generalized Poincar\'e duality
\cite{ih1,ih2})\label{thm_poincare_ih}
Let $X$ be a compact oriented pseudomanifold and let $p$ and $q$ be
perversities with $p+q=t$. Then:
$$I^p H^j(X)=I^q H^{m-j}(X).$$
\end{thm}

\medskip

\begin{exa} We will be interested in  the cases of the zero perversity $0=(0;\dots;0)$ and the maximal
perversity, which  are complement perversities.  By the above
theorem, we have for any pseudomanifold $X$ of
dimension $m$:
$$I^0 H^j(X)=I^t H^{m-j}(X).$$
\end{exa}

\medskip

\subsection{$L^\infty$-cohomology.}\label{sect_linfty}
We recall  the definition of the $L^\infty$ cohomology groups that
have been introduced by the author of the present paper in
\cite{vlinfty}. Let $M\subset \R^n$ be a smooth oriented submanifold.

\medskip

\begin{dfn}  We say that a form  $\omega$ on $M$
is $L^\infty$ if there exists a constant $C$ such that  for any $x
\in M$: $$|\omega(x)| \leq C. $$
 We denote by $\Omega^j_\infty
(M)$ the cochain complex constituted by all the $C^\infty$ $j$-forms
$\omega$ such that $\omega $ and $d\omega$ are both $L^\infty$.

The cohomology groups of this cochain complex are called the {\bf
$L^\infty$-cohomology groups of $M$} and will be denoted by
$H^\bullet _\infty(M)$. We may endow this cochain complex with the
norm:$$|\omega|_\infty:=\sup_{M} |\omega|+\sup_M |d\omega|. $$ We
also introduce the {\bf locally $L^\infty$ forms} as follows.
Given an open subset $U$  of $cl(M)$, let $\Omega_{\infty,loc}^j
(U\cap M)$ be the de Rham complex constituted by the smooth forms on $U\cap M$ locally bounded in $U$ which have a locally
bounded (in $U$) exterior derivative. This gives rise to a cohomology
theory that we shall denote $H^j _{\infty,loc}(U\cap M)$.

Similarly we define the de Rham complex $\overline{\Omega}_\infty ^\bullet (M)$ as the $L^\infty$ forms  weakly smooth and almost everywhere continuous. 
\end{dfn}

\medskip

By Theorem \ref{thm_intro_linfty}, we know that  the $L^\infty$
cohomology of a pseudomanifold coincides with its intersection
cohomology in the maximal perversity. We shall need the following theorem of \cite{vlinfty}.
 Set again $M^\ep = M\cap B^n(x_0;\ep)$ for some $x_0\in cl(M)$ fixed.

\bigskip

\begin{thm}
\label{thm_poincare}\cite{vlinfty}(Poincar\'{e} Lemma for
$L^\infty$ cohomology)
For $\ep$ positive small enough and any positive integer $j$:
$$H_\infty ^j(M^\ep)\simeq 0. $$
\end{thm}

\medskip

\subsection{Poincar\'e duality for $L^1$ cohomology}We give some corollaries  of Theorem \ref{thm_intro}. Thanks to  Goresky
and MacPherson generalized Poincar\'e duality, we get an explicit
topological criterion on the singularity to determine whether
$L^1$ cohomology is Poincar\'e dual to $L^\infty$ cohomology.

\medskip

\begin{cor}
Let $X$ be a compact oriented pseudomanifold. If $H^j(X_{reg})\simeq
I^0H^j(X)$ then $L^\infty$ cohomology is Poincar\'e dual to $L^1$
cohomology in dimension $j$, i. e. $$H^j_\infty(X_{reg})\simeq
H_{(1)}^{m-j}(X_{reg}).$$
\end{cor}
\begin{proof}
This  is a consequence of Theorems \ref{thm_intro},
\ref{thm_intro_linfty} and Goresky and MacPherson's generalized
Poincar\'e duality.
\end{proof}

\medskip

\begin{cor}Let $M \subset \R^n$ be an oriented  bounded $C^\infty$ submanifold.
 If  $\dim \delta M=k$ then $L^1$ cohomology is Poincar\'e dual to  $L^\infty$
cohomology in dimension $j<m-k-1$, i. e. for any positive integer  $j<m-k-1$:
$$H_{(1)} ^j(M) \simeq  H_\infty^{m-j} (M).$$
\end{cor}
\begin{proof} We may assume $k<m-1$ since otherwise the result is trivial. Set $X=cl(M)$ and observe that $X$ is a pseudomanifold.
Fix a Whitney $(b)$ stratification of $X$ (see \cite{bcr, tesc} for the construction of such stratifications) such that $X$ is a stratified pseudomanifold.  By definition of
$0$-allowable chains (see section \ref{sect_ih}), the
support of a singular chain $\sigma \in I^0 C_{j} (X)$  may not
intersect the strata of the singular locus of dimension less than
$m-j$.  If $j<m-k$ (and hence $k<m-j$) then there is no stratum of dimension bigger
or equal to $(m-j)$ and thus $|\sigma|$ must lie entirely in
$X_{reg}$ and therefore $$I^0C_{j} (X)=C_{j}(X_{reg}).$$

 Hence  if $j<m-k-1$, the same applies to $(j+1)$ and therefore $$I^0H_{j} (X)=H_{j}(X_{reg}).$$
 The result follows
from the preceding corollary.
\end{proof}

This corollary clearly implies Corollary
\ref{cor_poincare_duality_intro}.

\section{Lefschetz duality for $L^1$ cohomology.}\label{sect_dirichlet}
We are going to investigate Lefschetz duality. It means that we are going to consider $L^1$ forms satisfying boundary conditions. Our duality result will relate the cohomology of these forms to the cohomology of $L^\infty$ forms (Theorem \ref{thm_Poincare duality_dirichlet}). 

We first  define and study the de Rham complex of Dirichlet $L^1$ forms. In section \ref{sect_pd_dirichlet}, we establish Lefschetz duality for $L^1$ cohomology.

\subsection{Dirichlet $L^1$-cohomology groups.}\label{sect_Dirichlet_L^1-cohomology_groups.}
In this section,  $M$ is  an orientable  submanifold of $\R^n$ (not necessarily bounded) and $X$ will stand for its topological closure. We are going to consider $L^1$ forms with compact support. We recall that the support in $X$ of a $L^1$ form on $M$ is defined as the closure {\it in}  $X$ of the set of points at which this form is nonzero. Let $V\subset X$ be open.

\begin{dfn}
 We shall say that $\omega \in \ws _{(1)} ^j(M)$
has the {\bf $L^1$ Stokes' property in $V$} if for any $\alpha \in
\ws _{\infty,V}^{m-j-1}(M)$ we have:
\begin{equation}\label{eq_l1_stokes_property}\int_{M}
\omega\wedge \db \alpha=(-1)^{j+1} \int_{M} \db \omega\wedge
\alpha.\end{equation}

The de Rham complex of weakly smooth  $L^1$ forms of $M$  satisfying this property (and whose weak exterior derivative  satisfy this property as well) is called the complex of (weakly smooth) {\bf  Dirichlet $L^1$ forms on $M$} and is denoted $\ws_{(1)}^j(M ; V\cap  \delta M)$.  The subcomplex of the $C^\infty$ such forms is denoted  $\Omega_{(1)}^j(M ; V \cap  \delta M)$.

 As before, we denote by  $\ws_{(1),X}^j(M ; V \cap  \delta M)$ and  $\Omega_{(1),X}^j(M ; V \cap  \delta M)$ the subcomplexes of the forms having compact support in $X$.
\end{dfn}

\begin{rem}\label{rem_cpct_support}
If $\omega$ has compact support in $V$ and satisfies the $L^1$ Stokes' property in $V$ then clearly  (\ref{eq_l1_stokes_property})  holds for any  $\alpha \in
\ws _{\infty}^{m-j-1}(M)$. 
\end{rem}

 If $ K$ denotes a
compact manifold with boundary $\partial K$,  the relative de Rham complex
of differential forms $\Omega^j (K;\partial K)$ is  usually defined
as the set of $j$-forms $\omega$ on $K$ such that
$\omega_{|\partial K}\equiv 0$.   However, the smooth forms of the pair $(K;\partial K)$
may also be characterized as the smooth forms satisfying (\ref{eq_l1_stokes_property}) for
any smooth $L^\infty$ form $\alpha$ on $M$.
The Dirichlet $L^1$ cohomology defined above is therefore  completely analogous to the one of compact smooth manifolds.

In the case of non-compact manifolds,  it is not  possible to require that the forms vanish at the singularities  since the forms are not defined on
$\delta M$. If one wants  a similar characterization as in the case of compact manifolds with boundaries, we have to
require a condition near $\delta M$ and pass to the limit.

For this purpose, choose an exhaustion function $\rho:X\to \R^+$,
that is to say, a positive $C^2$ function  on $M$ tending
to zero as we approach $\delta M$. Then $\{\rho\geq
\varepsilon\}$ is a manifold with boundary $\{\rho=\ep\}$. Given
$\omega \in \ws_{(1)}^{j}(M)$, we may define an operator
on $\ws_{\infty,X} ^{m-j-1}(M)$
by:\begin{equation}\label{eq_def de lomega}l_\omega
(\alpha):=\lim_{\ep \to 0} \int_{\rho=\ep} \omega \wedge \alpha,
\end{equation}
for $\alpha \in \ws_{\infty,X} ^{m-j-1} (M)$. It is easy to see (by Stokes' formula)
that if $\alpha\in \ws_{\infty,X} ^{m-j-1} (M)$ and $\omega \in
\ws_{(1)} ^j (M)$ then the latter limit exists and that:
$$\int_{M} \omega\wedge \db \alpha=(-1)^{j+1}\int_{M}
\db \omega\wedge \alpha+l_\omega(\alpha).$$

 In particular the  limit in (\ref{eq_def de lomega}) is
independent of the exhaustion function $\rho$. Observe also that
$l_\omega$ is a bounded operator on $(\ws_{\infty,X} ^j
(M);|.|_\infty)$.

\medskip
\begin{dfn}\label{dfn_norme_partial}Set:$$|\omega|_{1,\delta}:=|l_\omega|,$$ where $|l_\omega|$
denotes the operator norm of $l_\omega$.\end{dfn}

Now, it follows from the definitions that if
$|\omega|_{1,\delta}=0$  if and only if the $L^1$ Stokes' property holds
for $\omega$. Hence, we get the following characterization of Dirichlet $L^1$ forms: \begin{equation}\label{eq_l1_rho}\ws_{(1)}
^\bullet (M;\delta M) = \{\omega \in
\ws_{(1)}^{\bullet}(M):
|\omega|_{1,\delta}=|\db \omega|_{1,\delta}=0 \}.\end{equation}

This characterization will be very useful to check that the $L^1$ Stokes property holds later on.

\begin{pro}\label{pro l1 isom smooth_dir}
The inclusions $\Omega_{(1),X} ^j(M;\delta M) \hookrightarrow
\overline{\Omega}_{(1),X} ^j(M;\delta M)$  and $\Omega_{(1)} ^j(M;\delta M) \hookrightarrow
\overline{\Omega}_{(1)} ^j(M;\delta M)$  induce isomorphisms in
cohomology.
\end{pro}
\begin{proof}
The argument used in the proof of Proposition \ref{pro l1 isom smooth}
also applies for Dirichlet cohomology.
\end{proof}

Let $\delta M^\ep:=\delta M \cap X^\ep$.  We can also make use of Proposition \ref{pro_neigh} in the same way as in section \ref{sect_cohomo_cpct_supp_local} to get the following exact sequence:
\begin{equation}\label{eq_long_dirichlet_local}
\dots \to  H^{j-1}_{(1)} (N^\ep;\delta N^\ep ) \to H^j_{(1),X^\ep} (M^\ep;\delta M^\ep )  \to H^j_{(1)} (M^\ep;\delta M^\ep ) \to \dots  \end{equation}

\subsection{Lefschetz duality for Dirichlet $L^1$ cohomology and the de Rham theorem.}\label{sect_pd_dirichlet}
 Let $M\subset \R^n$ be an orientable  submanifold of dimension $m$, set $X=cl(M)$ and take $x_0\in X$.
 Set again  $M^{\ep}:=M \cap B^n(x_0;\ep)$  and $N^\ep:=M\cap S^{n-1}(x_0;\ep)$.
  \subsection*{The operator $\K_0$.} We are going to construct a homotopy operator:
   $$\K_0:\overline{\Omega}_{(1)} ^m (M^\ep;\delta M^\ep) \to \overline{\Omega}_{(1)} ^{m-1} (M^\ep;\delta M^\ep),$$
    ($m=\dim M$)  based on the operator $\K_\nu$ introduced  in section \ref{sect_hom_hop}.

\begin{pro} On $\overline{\Omega}_{(1)} ^m (M^\ep)$: $$\lim_{\nu,t \to 0} |\K_\nu \omega -\K_t \omega|_{1} =0,
$$ 
and consequently $\lim_{\nu \to 0} \K_\nu$ 
defines   a homotopy operator $\K_0 :\overline{\Omega}^m _{(1)} (M^\ep)\to \overline{\Omega}^{m-1} _{(1)} (M^\ep)$.
\end{pro}
\begin{proof}
Let  $\omega \in \overline{\Omega}^m _{(1)} (M^\ep)$. Let $h$ be the homeomorphism used to define $\K_\nu$ (see  section \ref{sect_hom_hop}).  As $\omega$ is an $m$-form, $h^*\omega$ is $L^1$. Clearly we have:
$$ \lim_{t, \nu \to 0} \int_{M^\ep} |\K_t \omega -\K_\nu \omega| =\lim_{t, \nu \to 0, t\leq \nu} \int_t ^\nu  \int_{N^\ep} | \omega _2|=0,$$ 
since, as observed, $h^* \omega$ is $L^1$ on  $h^{-1} (M^\ep)$. 

As $\omega$  is an $m$-form, it is identically zero in restriction to $N^\nu$ since this is an $(m-1)$-dimensional manifold. Consequently
 $\pi_\nu^* \omega$ is zero and, as $\overline{d}\omega=0$,
by (\ref{eq_K_nu_homot_operator}) we have:
$$\overline{d}\K_\nu = Id_{ \Omega^m _{(1)} (M^\ep)}. $$
Passing to the limit we get that $\K_0\omega$ is weakly differentiable and that:
$$\overline{d}\K_0\omega =\omega, $$ as required.  
\end{proof}

\begin{pro}\label{pro_Komega_satisfait_l1}
Let $\omega \in \overline{\Omega}_{(1)}^j(M^\ep)$ satisfying the $L^1$
Stokes' property in $X^\ep$.\begin{enumerate} \item[(i)] If $0< j<m$ and $\omega$ has compact support in $X^\ep$
then $\K\omega$  satisfies the $L^1$ Stokes' property in $X^\ep$. \item[(ii)] If $j=m$, then $\K_0 \omega$ satisfies the
$L^1$ Stokes' property in $X^\ep$.\end{enumerate}
\end{pro}

\begin{proof} 
Let $\omega \in  \overline{\Omega}_{(1),X^\ep}^j(M^\ep)$ be a form satisfying the
$L^1$ Stokes' property. We have to check that $|\K
\omega|_{1,\delta}=0$ (see (\ref{eq_l1_rho})).

Consider a $C^2$ nonnegative
function $\rho_1 (x):N^\ep \to \R$ zero on $\delta N^\ep$ and positive on
$N^\ep$. Set $\rho_2=\rho_1 \circ h$ and  denote by $\rho$ the
Euclidian distance to $x_0$. For $\mu$ and $\nu$ positive real
numbers, let
$$M_{\mu,\nu}:= \{ x \in M^\ep: \rho_2(x)\geq \mu,\; \rho(x) \geq
\nu \}.$$ Then $M_{\mu,\nu}$ is a manifold with corners whose
boundary is the union of $\{x\in N^\nu: \rho_2(x)\geq \mu\} $ with
$$W_{\mu,\nu}=\{ x \in M^\ep : \rho_2(x)=\mu, \; \rho(x)\geq \nu \}.$$
Define $Z_{\mu,\nu}:=\partial W_{\mu\nu}$. Denote by $M_{\mu,\nu}'$, $W'_{\mu,\nu}$ and $Z'_{\mu,\nu}$ the respective  images by $h^{-1}$ of $M_{\mu,\nu}$, $W_{\mu,\nu}$ and $Z_{\mu,\nu}$. For the convenience of the reader, we gather all these notations on a picture:

\begin{figure}[ht]
\begin{pspicture}(0,-2)(5,5)

\psline[linewidth=1pt](-4.5,0)(-4.5,4)\psline[linewidth=1pt](.5,4)(.5,0)
\psline[linewidth=.5pt](-4,.8)(-4,4.3)\psline[linewidth=.5pt](0,.8)(0,4.3)

\psarc[linewidth=.5pt](-2,0){4.75}{58}{122}
\rput(0,-4){\psarc(-2,0){4.75}{58}{122}}
\rput(0,-3.5){\psarc[linewidth=.5pt](-2,0){4.75}{65}{115}}
\psdots[dotscale=1](-4,.84)(0,.84)(5.77,1.26)(6.24,1.26)

\psline[linewidth=.4pt,
arrowsize=1pt 2,
arrowlength=4,
arrowinset=0.2]
{->}(-2.45,1.85)(-3.8,1.05)
\psline[linewidth=.4pt,
arrowsize=1pt 2,
arrowlength=4,
arrowinset=0.2]
{->}(-1.5,1.85)(-0.15, 1.05)
\psline[linewidth=.4pt,
arrowsize=1pt 2,
arrowlength=4,
arrowinset=0.2]
{->}(-2.4,3.55)(-3.9, 3.3)
\psline[linewidth=.4pt,
arrowsize=1pt 2,
arrowlength=4,
arrowinset=0.2]
{->}(-1.4,3.55)(-0.15, 3.3)
\psline[linewidth=.4pt,
arrowsize=1pt 2,
arrowlength=2,
arrowinset=0.2]
{->}(-3.9,4.5)(-3.7, 4)


\rput(-2,-.5){$N^\ep \times [0;\ep)$}
\rput(-1.9,3.5){$W'_{\mu,\nu}$}
\rput(-4,4.8){$M'_{\mu,\nu}$}
\rput(-2,2){$Z'_{\mu,\nu}$}
\psline[linewidth=1pt]{->}(1,2)(3.5,2)\rput(2,2.5){$h$}

\psset{xunit=2cm}
\psset{yunit=4cm}
\rput(3,0){
\parametricplot[linewidth=1.2pt,plotstyle=curve]
 {-1}{1}{t t mul t mul 1.8 mul t t mul }}
\rput(3,0){\parametricplot[linewidth=.5pt,plotstyle=curve]
 {-.915}{-.4}{t t mul t mul 1.8 mul t t mul .15 add}}
\rput(3,0){\parametricplot[linewidth=.5pt,plotstyle=curve]
 {.915}{.4}{t t mul t mul 1.8 mul t t mul .15 add}}

\rput(3,0){
\parametricplot[linewidth=0.5pt,plotstyle=curve]
 {45}{135}{t cos 2.4 mul   t sin 1.2 mul }}
\rput(3,0){
\parametricplot[linewidth=0.5pt,plotstyle=curve]
 {80}{100}{t cos .65 mul   t sin .32 mul }}

\psline[linewidth=.7pt,
arrowsize=1pt 1,
arrowlength=3,
arrowinset=0.3]
{->}(2.92,.45)(2.89,.35)
\psline[linewidth=.7pt,
arrowsize=1pt 1,
arrowlength=3,
arrowinset=0.3]
{->}(3.07,.45)(3.1,.35)
\psline[linewidth=.2pt,
arrowsize=1pt 2,
arrowlength=4,
arrowinset=0.2]
{->}(2.8,.8)(2.4, .7)
\psline[linewidth=.2pt,
arrowsize=1pt 2,
arrowlength=4,
arrowinset=0.2]
{->}(3.2,.8)(3.6, .7)
\psline[linewidth=.4pt,
arrowsize=1pt 2,
arrowlength=2,
arrowinset=0.2]
{->}(2.05,1.15)(2.15, 1)
\psline[linewidth=.7pt,
arrowsize=1pt 2,
arrowlength=2,
arrowinset=0.2]
{->}(3.5,.25)(3.43, .37)
\psline[linewidth=.2pt,
arrowsize=1pt 2,
arrowlength=2,
arrowinset=0.2]
{->}(3.85,1.19)(3.7, 1.16)

\rput(4,1.2){$N^\ep$}
\rput(3.5,.2){$\delta M^\ep$}
\rput(3,.85){$W_{\mu,\nu}$}
\rput(2,1.2){$M_{\mu,\nu}$}
\rput(3,.5){$Z_{\mu,\nu}$}

\end{pspicture}
\caption{ The Lipschitz conic structure of $M^\ep$. Here $Z_{\mu,\nu}$ and  $Z_{\mu,\nu}'$ are reduced to two points. }
\end{figure}
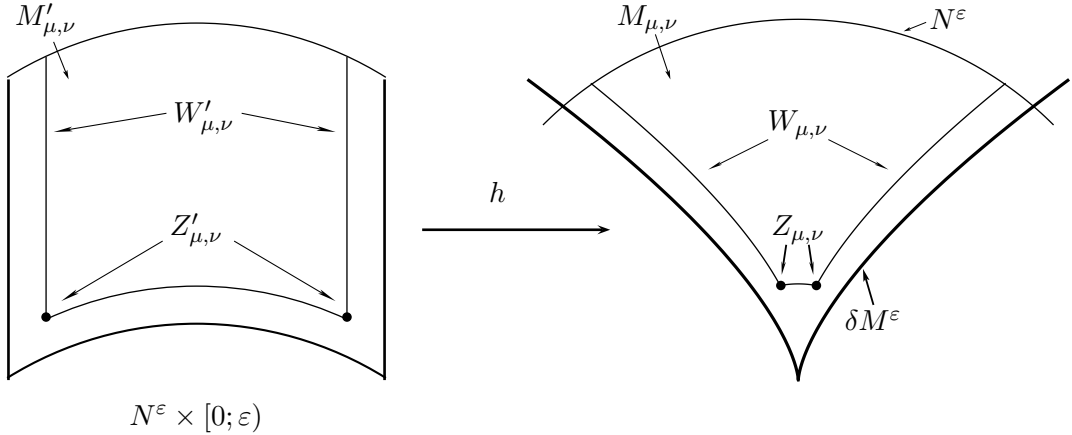

 Observe that by construction (recall that $\rho(h(t;x))=t$) we have  $W'_{\mu,\nu}=Z'_{\mu,\nu} \times [\mu;\ep)$.

By Proposition \ref{pro_proprietes_de_K}, we already know that:$$
\lim_{t \to 0} \int_{N^t} |\K \omega| =0.$$ Therefore it is enough
to check that for every positive real number $\nu$:
\begin{equation}\label{eq_lim_X_K}\lim_{\mu \to
0_+} \int_{ W_{\mu,\nu}} \K\omega \wedge \alpha =0,
\end{equation}
for any $\alpha \in \ws^{m-j-1} _{\infty,X^\ep}(M^\ep)$.

Fix such a form $\alpha$. Write $\beta =h^*\alpha$ for simplicity, and decompose 
$\beta=\beta_1+dt \wedge  \beta_2$ as well as $h^*\omega=\omega_1+dt \wedge  \omega_2$.  Observe that:\begin{equation}\label{eq_beta_1_et_omega_2} \beta_1\wedge \omega_2=0 \qquad \mbox{on } W'_{\mu,\nu},\end{equation}
since this differential $(m-1)$-form does not involve $dt$.
\begin{eqnarray}\label{eq_K_adj_de_K'}
\int_{W_{\mu,\nu}}  \K \omega \wedge \alpha&=& \int_{(t;x) \in W'_{\mu,\ep}} (\int_{s=t} ^\ep \omega_2(s;x) ds ) \wedge \beta(t;x)\nonumber\\
&=&  \int_{x\in Z_{\mu,\ep}'} \int_{t=\nu} ^\ep \int_{s=t} ^\ep \omega_2(x;s)\wedge \beta_2(t;x) \,ds\, dt\nonumber \quad  \mbox{(by (\ref{eq_beta_1_et_omega_2}))}  \\
&=&  \int_{Z_{\mu,\ep}'} \int_{s=\nu} ^\ep \int_{t=\nu} ^s \omega_2(s;x)\wedge \beta_2(t;x) \,dt \,ds \quad \mbox{(by Fubini)}\nonumber\\
&=&  \int_{s=\nu} ^\ep\int_{Z_{\mu,\nu}'}   h^*\omega(x;s) \wedge
\int_{t=\nu} ^s\beta_2(t;x)\,dt.
\end{eqnarray}

 Define a  form $\K'_\nu\alpha$ on $ (0;\ep)\times N^\ep$ by
$$\K'_\nu \alpha(s;x):=\int_{t=\nu} ^s\beta_2(t;x)\,dt$$ if $s\geq \nu$, and set $\K'_\nu\alpha (s;x)$ to be zero if $s\leq \nu$. By $(2)$ of definition \ref{dfn conical}, $h$ induces a quasi-isometry on $[\nu;\ep)\times N^\ep$ (see Remark \ref{rem_conical} $(3)$) and therefore $h^{-1*}\K'_\nu\alpha$ is an $L^\infty$ form. 
 Moreover, in view of (\ref{eq_K_adj_de_K'}), we clearly have:
\begin{equation}\label{eq_K_K'_adj_enonce}\int_{W_{\mu,\nu}} \K \omega \wedge \alpha= \int_{W_{\mu,\nu}'}  h^*\omega \wedge \K'_\nu \alpha.\end{equation}

Now, as by definition $\K'_\nu \alpha$ is zero on $\partial M_{\mu,\nu}'\setminus W_{\mu,\nu}'$, this amounts to:
$$ \int_{ W_{\mu,\nu}} \K  \omega \wedge
 \alpha = \int_{\partial M_{\mu,\nu}} \omega
\wedge h^{-1*}\K' _\nu \alpha'$$ which tends to zero as $\mu$ goes to zero for
$\omega$ satisfies the $L^1$ Stokes property and $\K'_\nu\alpha$ is an
$L^\infty$ form (see Remark \ref{rem_cpct_support}), yielding (\ref{eq_lim_X_K}) and establishing $(i)$.

\medskip

For a proof of $(ii)$,
observe that for any $L^\infty$ $(m-j-1)$-form $\alpha$ with compact support in $X^\ep$:
$$ \lim_{t \to 0}\int_{N^t} |\K_0 \omega \wedge \alpha|\leq  C \lim_{t \to 0}
\int_{ (0;t)\times N^\ep} |h^*\omega |=0  $$ (with $C=\sup |\alpha|$).
%

Therefore, like in the proof of $(i)$, it is enough to show (\ref{eq_lim_X_K}) for $\K_0$. By
definition,  $\K \omega$ is an $(m-1)$-form with no
differential term involving $dt$. Thus $\K_0 \omega$ must be
identically zero on $W_{\mu,\nu}$ and consequently
(\ref{eq_lim_X_K}) is trivial in this case.
\end{proof}

\begin{pro}\label{pro_poinc_lemma_dirichlet} (Poincar\'e Lemma for Dirichlet $L^1$ cohomology)
For $j<m$ and $\ep >0$ small enough   $$H_{(1),X^\ep} ^j (M^\ep;\delta M ^\ep)
\simeq 0\simeq H_{(1)} ^m(M^\ep;\delta M^\ep) .$$
\end{pro}
\begin{proof}
The case $j=0$ is clear. Let   $0<j<m$ and let  $\omega \in \Omega_{(1),X^\ep}
^j(M^\ep;\delta M ^\ep)$ be a closed form. Then, by the preceding
proposition $\K \omega$ satisfies the $L^1$ Stokes' property. Furthermore, $\overline{d} \K \omega=\omega$ and, by  Proposition \ref{pro l1 isom smooth_dir}, $\K \omega$ satisfies the $L^1$ Stokes' property. The first isomorphism ensues.

To compute $H_{(1)} ^m(M^\ep;\delta M ^\ep)$, just use $\K_0$ and $(ii)$ of
the preceding proposition exactly in the same way.
\end{proof}

\subsection*{Lefschetz duality for $L^1$ cohomology.} The setting is still the same as in section \ref{sect_pd_dirichlet}.

\begin{thm}\label{thm_Poincare duality_dirichlet}  The pairing
$$H_{(1),X}^{m-j}(M;\delta M)  \otimes H_{\infty,loc} ^{j}(M)\to \R$$
$$(\alpha;\beta) \mapsto \int_{M}\alpha \wedge \beta $$
is  nondegenerate.  \end{thm}

By ``nondegenerate'' we mean that for any $L^1$ differential form
with compact support $\beta$ there is a locally $L^\infty$
differential form $\alpha$ such that $\int_M\alpha \wedge
\beta =1$ and for any closed locally $L^\infty$ form $\alpha$
there is a  form $\beta \in \Omega_{(1),X} ^{m-j}(M;\delta M)$ for which the latter integral is
nonzero as well.

\begin{proof}
We shall apply an argument which is similar to the one  used in
the proof of Theorem \ref{thm_l1_psi_M_isom}. As we may argue by
induction on $m$, we  shall assume that the theorem holds for
 manifolds of dimension $(m-1)$, $m\geq 1$.

Consider the complex of
presheaves on $X$ defined by  $\Omega_{(1),U}^j(U\cap
M;U\cap \delta M)^*$ (where $*$ denotes the algebraic dual vector space), if $U$ is an open subset of $X$, and denote by
$\la_{(1)}^j$ the resulting differential sheaf. Let $\hn^\bullet
(\la_{(1)}^\bullet)$ be the derived sheaves. Similarly,  denote by
$\la^j_{\infty}$ the differential sheaf resulting from the presheaf
$\Omega_{\infty,loc}^j(U\cap M)$.

For every subset $U\subset X$ and $j\leq m$, consider the  mappings
$$\varphi_U ^j :  \Omega_{\infty,loc}^{j}(U\cap M)\to \Omega_{(1),U}^{m-j}(M\cap U;\delta M\cap U)^*,$$
defined by $\varphi_U^j (\alpha):\beta \mapsto \int_{U\cap M} \alpha \wedge
\beta$. It
 follows from the theory of spectral sequences (see for instance \cite{bredon} IV Theorem  2.2) that, if the mapping of complex of differential sheaves induced by
  $\varphi_U ^j$ is a local isomorphism,  then $\varphi_M^j$ induces an isomorphism between the  cohomology groups of the respective global sections
 of  $\la_{(1)}^{m-j}$  and $\la_\infty ^{j}$, as required.

 Thus, we  simply have to make sure that
 the mappings $\varphi_U ^j$'s induce local isomorphisms at any $x_0 \in cl(M)$.  Notice that by
 Theorem \ref{thm_poincare} and Proposition \ref{pro_poinc_lemma_dirichlet}, this is clear for $j>0$.

It remains to deal with the case where $j=0$. 

 As we can
work separately on the connected components of $M^\ep$ we will assume that $M^\ep$ is connected. By Theorem \ref{thm_Poincare duality_dirichlet} we have: $$H^m_{(1)}(M^\ep ;\delta M^\ep) \simeq 0.$$ 
 By induction on the dimension, we know that Lefschetz duality   holds for $N^\ep$.  Since  $N^\ep$ is connected, by Theorem \ref{thm_poincare_ih} we get: $$H^{m-1}_{(1)}(N^\ep;\delta N^\ep) \simeq H_\infty ^{m-1}(N^\ep)\simeq I^t H^{m-1}(N^\ep) \simeq \R,$$
 (see \cite{ih1,ih2} for the local computations of the intersection homology groups). 

 Thanks to the long exact sequence (\ref{eq_long_dirichlet_local}), we deduce that:$$H^m_{(1),X^\ep}(M^\ep ;\delta M^\ep)\simeq H^{m-1}_{(1)}(N^\ep;\delta N^\ep) \simeq \R. $$


 Hence, it is enough to show that $\varphi^m_{M^{\ep}}$ is onto.  As the $L^\infty$ closed  $0$-forms are reduced to the constant
 forms, it suffices  to prove that for $x_0 \in cl(M)$ and $\ep
 >0$ small enough,  
 we can find $\omega \in \Omega_{(1),X^\ep} ^m (M^\ep;\delta M^\ep)$ such that  $\int _{M^\ep} \omega\neq 0$.

 As $M^\ep$ is orientable we can find a volume form on $M^\ep$. We may multiply this form by a bump function to get a form with compact support in $X^\ep$. The integral on $M^\ep$ of this form is then necessarily nonzero.  This shows that $\varphi^m_{M^{\ep}}$ is onto.
\end{proof}

Of course, when $M$ is bounded, $H_{\infty, loc}^j (M)$ (resp.
$H_{(1),X}^{j}(M;\delta M)$) and $H_{\infty} ^j(M)$ (resp.
$H_{(1)}^{j}(M;\delta M)$) coincide so that the latter paring
induces in the case of bounded manifold an isomorphism between
$H_\infty ^j (M)$ and the dual vector space of $H_{(1)} ^j
(M;\delta M)$, establishing Theorem
\ref{thm_intro_poincare_dirichlet}.

\begin{rem}As explained in the introduction, Theorem \ref{thm_intro_poincare_dirichlet} and Generalized Poincar\'e duality  imply the de Rham theorem for Dirichlet $L^1$ cohomology (Corollary \ref{cor_dirichlet_de_rahm_intro}). In this section we assumed that $M$ is orientable. This is necessary to prove Lefschetz duality for $L^1$ cohomology (Theorem \ref{thm_Poincare duality_dirichlet}). Nevertheless, the de Rham theorem for $L^1$ cohomology could be proved directly (independently of Lefschetz duality) and then orientability is unnecessary. 
\end{rem}


\section{On the $L^1$ Stokes' property}\label{sect_l1sp} Let $M\subset \R^n$ be a bounded  orientable submanifold.  The latter theorem raises a natural question:
when do we have the $L^1$ Stokes' property on a subanalytic
manifold ? This amounts to wonder when the Dirichlet $L^1$
forms and the $L^1$ forms coincide not only as cohomology groups,
but also as cochains complexes. The following theorem answers very
explicitly. The $L^1$ Stokes' property holds for  $j$-forms iff
 $\delta M$ is of dimension less than $(m-j-1)$.

In particular, if a subanalytic compact set $X\subset \R^n$ has only isolated singularities, then
the $L^1$ Stokes' property holds for any $L^1$ $j$-form on $X_{reg}$, $j<m-1$. Below we
adopt the convention that $\dim \emptyset =-1$.

\begin{thm} Let $j<m$.
The $L^1$ Stokes' property holds for $j$-forms iff $\dim \delta M
< m-j-1$. In this case,  $L^1$ cohomology is naturally dual to
$L^\infty$ cohomology in dimension $j$, i. e. the pairing:
$$ H_{(1)}^{j}(M) \otimes H_\infty ^{m-j}(M) \to \R$$
$$(\alpha;\beta) \mapsto \int_{M}\alpha \wedge \beta $$
is (well defined and) nondegenerate.
\end{thm}

\begin{proof}
We first focus on the if part. Write $X:=cl(M)$.

 As pointed out in section
\ref{sect_Dirichlet_L^1-cohomology_groups.} (see
(\ref{eq_l1_rho})), it is enough to show that for any $\omega \in  \overline{\Omega}_{(1)}^j (M)$ we have
$|\omega|_{1,\delta}=0$. We shall prove by induction the following statements.

\bigskip

{\bf$(\textrm{A}_k)$}  Let $a<b$ be real numbers and let $k$ and $l$ be  integers.  Let $M$ be a bounded manifold with $\dim \delta M=k$. Set $\mathbb{D}:=[a;b]^l$. Write $\overline{\Omega}_{(1),X\times \dbb}^j (M \times \dbb)$ for the weakly smooth forms $\omega$ on $M \times \dbb$, with compact support in $X\times \dbb$, such that  $\omega$ and $\db \omega$ are continuous near almost every point of $M \times \pa \dbb$ and $L^1$ on $M \times  \dbb $ and on  $M \times \pa \dbb $.

 Let $\theta:X\to \R$ be a $C^2$ nonegative function with $\theta^{-1}(0)=\delta M$.
For $\omega \in \overline{\Omega}_{(1),X\times \dbb}^j (M \times \dbb)$ and 
$\alpha \in \overline{\Omega}_{\infty}^{m-j+l-1} (M \times \dbb)$ we have:
$$\lim_{\nu \to 0} \int_{\{\theta=\nu\} \times \dbb} \omega \wedge \alpha =  0.$$

\medskip
The 'if part' of the theorem follows from the case where $l$ is zero.  The product by $\dbb$ will be useful to perform the induction step. 
Note that the case where $\dim \delta M=-1$ is obvious since in this case $\{\theta=\nu\}$ is empty for $\nu$ small enough.

Fix $\omega$ and $\alpha$ like in {\bf$(\textrm{A}_k)$}, $k\geq 0$.  It suffices to prove {\bf$(\textrm{A}_k)$} 
 for the forms $\varphi_i\omega$, if $\varphi_i$ is a
partition of unity. This means that we can work locally and assume
that the support of  $\omega$ in $X$  is included in a little ball $B^n(x_0;\ep)\times \dbb$ with
$\ep>0$ and $x_0\in X$.

We adopt the same notations as in the proof of Proposition
\ref{pro_Komega_satisfait_l1} that we recall (see fig. $1$). Consider a $C^2$ nonnegative function $\rho_1 (x):N^\ep\to
\R$ zero on $\delta N^\ep$ and positive on $N^\ep$. Set $\rho_2=\rho_1
\circ h^{-1}$ (recall that $h$ is the local mapping provided by Theorem \ref{thm Lipschitz conic structure}) and denote by $\rho$ the Euclidian distance to $x_0$. For
$\mu$ and $\nu$ positive real numbers, let
$$M_{\mu,\nu}:= \{ x \in M^\ep: \rho_2(x)\geq \mu,\; \rho(x) \geq
\nu \}.$$ Then $M_{\mu,\nu}$ is a manifold with corners (for $\mu$ and $\nu$ generic) whose
boundary is the union of the set $\{x\in N^\nu: \rho_2(x)\geq \mu\} $ with the set
$$W_{\mu,\nu}=\{ x \in M^\ep : \rho_2(x)=\mu, \; \rho(x)\geq \nu \}.$$
Denote by $Z_{\mu}$ the set $\{ x\in N^\ep:    \rho_2(x)=\mu\}$. 

We shall  show that \begin{equation}\label{eq_lim_X} \lim_{\nu\to
0}\,\lim_{\mu \to 0} \int_{\partial M_{\mu,\nu}\times \dbb} \omega \wedge
\alpha =0.
\end{equation}

 Extend trivially  the mapping $h$ to a mapping  $h':N^\ep \times [0;\ep] \times \dbb \to M^\ep \times \dbb$ and let  $\omega':=h'^*(\omega)$ and $ \alpha':=h'^*(\alpha)$. Note  that as $h^{-1}( W_{\mu,\nu})=Z_\mu \times [\nu;\ep]$:
$$\lim_{\mu \to 0} \int_{W_{\mu,\nu}\times \dbb } \omega\wedge \alpha =  \lim_{\mu \to 0} \int_{Z_{\mu}\times [\nu;\ep]\times \dbb  } \omega'\wedge \alpha',$$
which tends to zero  thanks to the induction hypothesis (since $\dim \delta N^\ep<k$). It thus remains to show that:
\begin{equation}\label{eq_lim_Y} \lim_{\nu \to 0} \int_{N^\nu\times \dbb}  \omega \wedge
\alpha = 0. \end{equation}

 We shall again make use of the
homotopy operator $\K$. We extend $\K$ to a operator on $\overline{\Omega}_{(1),X\times \dbb}^{j+l} (M^\ep \times \dbb)$, considering the extra variables in $\dbb$  as parameters (if a form $\omega(x;t)$ on $M^\ep \times \dbb$ is $L^1$ then the form $\omega_t(x):=\omega (x;t)$ is $L^1$ on $M^\ep$ for almost every $t\in\dbb$).  For almost every $t$, $\K \omega_t$ is a $L^1$ form of $M^\ep$. Moreover, by remark \ref{rem_K}, the forms $\beta(x;t):=\K \omega_t(x)$ and $\beta'(x;t):= \K d\omega_t (x)$ are $L^1$ forms on $M^\ep \times \dbb$. Then (\ref{eq_K_homot_operator}) continue to hold for $L^1$ forms with compact support in $X^\ep \times \dbb$.

This   identity entails that (\ref{eq_lim_Y}) splits  into:
\begin{equation}\label{eq_lim_YK}\lim_{\nu \to 0}  \int_{N^\nu\times \dbb} \K \db\omega_t \wedge \alpha=0.\end{equation}
and
\begin{equation}\label{eq_lim_YD}\lim_{\nu \to 0}\int_{N^\nu\times \dbb}\overline{d}\K \omega_t \wedge \alpha=0.\end{equation}
In virtue of {\bf$(\textrm{A}_{k-1})$}  the $L^1$ Stokes' property holds on $N^\nu\times \dbb$ and, integrating by parts, the latter equation may be rewritten as:
\begin{equation}\label{eq_lim_YDbis}\lim_{\nu \to 0}\;[ \int_{N^\nu\times \dbb} \K \omega \wedge \overline{d}\alpha+ \int_{N^\nu\times \pa \dbb} \K \omega \wedge \alpha]\;=0.\end{equation}

Observe that  (\ref{eq_int_komega_link}) holds for $\omega_t$ and  $\db\omega_t$ for almost every $t$, i. e. that we have for almost every $t$ in $\dbb$: $$\lim_{\nu \to 0} \int_{N^\nu}|\K \omega_t|=\lim_{\nu \to 0} \int_{N^\nu}|\K\db\omega_t|=0.$$
Therefore,  as $\alpha$ and $\db \alpha$ are $L^\infty$, (\ref{eq_lim_YK}) and (\ref{eq_lim_YD}) (via (\ref{eq_lim_YDbis})) both come down from the Lebesgue dominated convergence theorem.

For the statement on Poincar\'e duality, observe now that the condition $\dim \delta M <m-j-1$ ensures that  $(j-1)$ and $j$ forms satisfy the $L^1$ Stokes' property.
Hence,  $$H_{(1)} ^j (M) \simeq H_{(1)} ^j(M;\delta M)$$ and the statement follows from Theorem \ref{thm_Poincare duality_dirichlet}.

\bigskip

It remains to prove that if the $L^1$ Stokes' property holds for all $j$-forms then $\dim \delta M < m-j-1$. Fix $j<m$. We shall indeed establish the contraposition. 

 Let $k:=\dim \delta M$. Assume $k \geq m-j-1$ and take a regular point $x_0$ of $\delta M$.

Up to a local diffeomorphism we may identify a neighborhood $W$ of $x_0$ in $\delta M$ with an open subset of $\R^{k}$ (that we will still denote $W$). Also, thanks to subanalytic  bi-Lipschitz triviality \cite{vlt}, there is a subanalytic by-Lipschitz map $H$ sending  a contractible neighborhood $U$ of $x_0$ in $X$ onto a product   $W \times X'$, with $X'$ having only an isolated singularity. We can also assume that $H(M\cap U)$ is a product $W \times M'$.   
   
By Proposition \ref{prop_pullback_weakly smooth forms}, subanalytic bi-Lipschitz maps induce a one-to-one correpondence between  weakly smooth forms and consequently,  $M$ satifies the  $L^1$ Stokes' property  if and only if so does $W \times M'$.   Therefore it is enough to show the result on  $ W\times M' $.

Observe that $$H_{(1),X'}^{m-k} (M') \simeq  0,$$
while  $H_{(1),X'}^{m-k} (M';\delta M \cap X')$ is nonzero (by Corollary  \ref{cor_dirichlet_de_rahm_intro}). Consequently there must be a form $\omega \in \Omega_{(1),X'} ^{(m-1-k)} (M')$ which does not satisify the  $L^1$ Stokes' property.  Define an $L^1$ $j$-form on $M$ by:    $$\alpha:=\omega \wedge d x_1 \wedge \dots \wedge d x _{j-m+k+1},$$ 
 where $d x_1, \dots ,d x_k$ is the canonical basis of $1$-forms on $W$ ($(j-m+k+1)$ is nonnegative by assumption). We claim that $\alpha$ does not satisfy the  $L^1$ Stokes' property in $W \times X'$.   We will exhibit a form $\beta \in \Omega_{\infty, W \times X'} ^{m-j-1} (W \times M')$ such that $l_\alpha (\beta)\neq 0$.

For this purpose, recall that since the  $L^1$ Stokes'  property fails for $\omega$ on $M'$, there exists a form $\theta \in \ws_{\infty , X'} ^{0}(M')$ for which $l_\omega (\theta) \neq 0.$  Define a form on $W\times M'$ by:
$$\theta ':=\theta \,d x _{j-m+k+2} \wedge \dots \wedge d x _{m}.$$  
As $\theta '$ does not have compact support in $W \times X'$,  we shall multiply it by a bump function. Let $\psi: W \to [0;1] $ be a smooth nonegative compactly supported function which takes value $1$ at $x_0$ and set $\beta:=\psi \theta '$. By Fubini   $$l_\omega (\beta) =l_\omega (\theta)\int_{W} \psi(y)dy \neq 0,$$ as required.  
\end{proof}

\begin{rem}
The argument used in the above proof was essentially local. Therefore, if we replace $L^\infty$ by $L^\infty_{loc}$ and $L^1$ by $L^1$ with compact support in $X$ the theorem goes over unbounded manifolds as well.
\end{rem}

\section{An example.}
We end this paper by an example on which we discuss all the
results of this paper. Let $X$ be the suspension of the torus.





 This is the set
constituted by two cones over a torus that are attached along this
torus. It is the most basic example on which Poincar\'e duality
fails for singular homology but holds for intersection homology \cite{ih1}.  Let $x_0$ and $x_1$ be the two
isolated singular points.

This set is a pseudomanifold. It has very simple singularities
(metrically conical). However, the results of this paper show that
if they were not conical (say cuspidal), this would not affect the
cohomology groups which only depend on the topology of the
underlying singular space.  This simple example is already enough to illustrate  how
the singularities affect Poincar\'e duality for $L^1$ cohomology.

The different cohomology groups
considered in this paper are gathered in the table below.

 \begin{center}
\begin{tabular}{|l|c|c|c|c|}
  \hline 

  {Cohomology groups}{$\qquad\qquad\quad j=$}
 & $0$& $1$& $2$ & $3$  \\
 \hline $ I^t H^j (X)$ and $ H_\infty ^j (X_{reg})$ & $\R$ & $0$  & $\R^2$  & $\R$     \\
  \hline $ I^0 H^j (X)$ and $ H_{(1)} ^j (X_{reg};X_{sing})$   & $\R$  &$\R^2$  & $0$   &$\R$   \\
  \hline  $H^j(X_{reg})$ and $H_{(1)}^j (X_{reg})$& $\R$   & $\R^2$  & $\R$  & $0$    \\
  \hline
\end{tabular}
\end{center}
\medskip

All these results may be obtained from the isomorphisms given in
the introduction and a triangulation. Below, we interpret them geometrically.

Let $T \subset X$ be the original torus  and let $\sigma$ and
$\tau$ be the suspension of the (support of the) two generators of
$H_1 (T)$. Write $\sigma^\ep :=\{x \in |\sigma|: d(x;\{x_0,x_1\})
\leq \ep \} $.

If $\omega$ is an $L^\infty$ $2$-form zero near the singular points
and satisfying
\begin{equation}\label{eq_ex_sigma}\int_\sigma \omega=1,\end{equation} and if $\omega=d
\alpha$ then $\int_{\sigma^\ep} \alpha\equiv 1$ (by Stokes' formula). As
the volume of $\sigma^\ep$ tends to zero, $\alpha$ cannot be
bounded. Consequently if $\omega$ is a $L^\infty$ closed $2$-form
zero near the singularities satisfying (\ref{eq_ex_sigma}), it
must represent a nontrivial class. In fact, every nontrivial class
may be represented by a shadow form \cite{bgm}.

However, the form $\alpha$ may be $L^1$. The only nontrivial $L^1$ class of
$2$ forms is actually provided by those forms whose integral  on $T$ is
nonzero, but these forms obviously do not satisfy the $L^1$
Stokes property (see (\ref{eq_l1_rho})). We see that the
singularities induce a gap between $L^1$ and Dirichlet $L^1$
cohomology, making the $L^1$ Stokes' property fail.

We also see that $L^\infty$ cohomology is dual to $L^1$ cohomology
in dimension $2$ and $3$ (as it is established by Theorem
\ref{cor_poincare_duality_intro}).
 However, $H_{\infty} ^1
(X_{reg})$ is not isomorphic to $H_{(1)}^{2} (X_{reg})$.


\begin{thebibliography}{mmm}
\bibitem [BCR]{bcr} J. Bochnak, M.Coste and M.-F.Roy, Real Algebraic Geometry. Ergebnisse der Math. 36, Springer-Verlag(1998).
\bibitem[B]{bredon} G.  E. Bredon, Sheaf theory. Second edition. Graduate Texts in Mathematics, 170. Springer-Verlag, New York, 1997. xii+502 pp.
\bibitem[BT]{bl}R. Bott,  Tu W. Loring , Differential forms in algebraic topology. Graduate Texts in Mathematics, 82. Springer-Verlag, New York-Berlin, 1982.
\bibitem[C1]{c1}J. Cheeger,  Hodge theory of complex cones.  Analysis and topology on singular spaces, II, III (Luminy, 1981),  118--134, Ast\'erisque, 101-102, Soc. Math. France, Paris, 1983.
\bibitem[C2]{c2} J. Cheeger, On the spectral geometry of spaces with cone-like singularities.  Proc. Nat. Acad. Sci. U.S.A.  76  (1979), no. 5, 2103--2106.
\bibitem[C3]{c3}J.  Cheeger, On the Hodge theory of Riemannian
   pseudomanifolds. Geometry of the Laplace operator (Proc. Sympos. Pure Math.,
   Univ. Hawaii, Honolulu, Hawaii, 1979), pp. 91--146, Proc. Sympos. Pure
   Math., XXXVI, Amer. Math. Soc., Providence, R.I., 1980.
 \bibitem[C4]{c4}J. Cheeger,  Hodge theory of complex cones.  Analysis and topology on singular spaces, II, III (Luminy, 1981),  118--134, Ast\'erisque, 101-102, Soc. Math. France, Paris, 1983.
\bibitem[C5]{c5} J.  Cheeger, Spectral geometry of singular Riemannian spaces.
/J. Differential Geom. (1983), 575--657 (1984).
\bibitem[BGM]{bgm} J.-P. Brasselet, M. Goresky, R. MacPherson,
Simplicial differential forms with poles.
Amer. J. Math. 113 (1991), no. 6, 1019--1052.
\bibitem[CGM]{cgm}J. Cheeger, M. Goresky, R. MacPherson,  $L\sp{2}$-cohomology and intersection homology of singular algebraic varieties.
\bibitem[vDS]{ds} L. van den Dries, P. Speissegger.
 P. O-minimal preparation theorems.  Model theory and applications,  87--116, Quad. Mat., 11, Aracne, Rome, 2002.
  \bibitem[DS]{tesc}  	Z. Denkowska, J. Stasica, Ensembles sous-analytiques \`a la Polonaise, Editions Hermann, Paris 2007.
 \bibitem[D]{d} J.  Dodziuk,
Sobolev spaces of differential forms and de Rham Hodge isomorphism.
J. Differential Geom. 16 (1981), no. 1, 63--73.
\bibitem[GM1]{ih1}M. Goresky, R.   MacPherson, Intersection homology theory.  Topology  19  (1980), no. 2, 135--162.
\bibitem[GM2]{ih2}M. Goresky, R.   MacPherson, Intersection homology. II.  Invent. Math.  72  (1983), no. 1, 77--129.
\bibitem[HP]{hp} W. C. Hsiang, V. Pati,
$L\sp 2$-cohomology of normal algebraic surfaces. I.
Invent. Math. 81 (1985), no. 3, 395--412.
\bibitem[IM]{iw} T. Iwaniec, G. Martin,
Geometric function theory and non-linear analysis.
Oxford Mathematical Monographs. The Clarendon Press, Oxford University Press, New York, 2001.
\bibitem[MT]{mt} C. Murolo, D. Trotman, Horizontally-$C^1$ controlled stratified maps and Thom's first isotopy theorem.  C. R. Acad. Sci. Paris Sér. I Math.  330  (2000),  no. 8, 707--712.
\bibitem[L2]{l}S. \L ojasiewicz,
Th\'eor\`eme de Paw\l ucki. La formule de Stokes sous-analytique.
Geometry Seminars, 1988--1991 (Italian) (Bologna, 1988--1991),
79--82, Univ. Stud. Bologna, Bologna, 1991.
\bibitem[P]{p}A. Parusi\'nski, Lipschitz stratification of subanalytic sets.
 Ann. Sci. Ecole Norm. Sup. (4) 27 (1994), no. 6, 661--696.
\bibitem[S1]{s1}L. Saper,  $L_ 2$-cohomology and intersection homology of certain algebraic varieties with isolated singularities.  Invent. Math.  82  (1985),  no. 2, 207--255.
\bibitem[S2]{s2}L. Saper,
$L\sb 2$-cohomology of K\"ahler varieties with isolated singularities.
J. Differential Geom. 36 (1992), no. 1, 89--161.
\bibitem[SV]{sv} L. Shartser, G. Valette, De Rham theorem for $L^\infty$ cohomology on singular spaces, preprint.
\bibitem[V1]{vlt} G. Valette, Lipschitz triangulations. Illinois J. Math. 49 (2005), issue 3,
953--979.
\bibitem[V2]{vpams}G. Valette, , The link of the germ of a semi-algebraic metric space  Proc. Amer. Math. Soc. 135 (2007), 3083--3090.
\bibitem[V3]{vlinfty} G. Valette, $L^\infty$ cohomology is
intersection cohomology, prerpint.
 \bibitem[Wa]{w}F. Warner, Foundations of differentiable manifolds and Lie groups. Corrected reprint of the 1971 edition. Graduate Texts in Mathematics, 94. Springer-Verlag, New York-Berlin, 1983.
\bibitem[We]{weber} A. Weber, An isomorphism from intersection homology to $L_p$-cohomology.  Forum Math.  7  (1995),  no. 4, 489--512.
\bibitem[Wh]{wh} H. Whitney,  Geometric integration theory. Princeton University Press, Princeton, N. J., 1957. \bibitem[Y]{y} B. Youssin,  $L\sp p$ cohomology of cones and horns.  J. Differential Geom.  39  (1994),  no. 3, 559--603.
\end{thebibliography}
\end{document}